\documentclass[12pt,reqno]{amsart}
\usepackage{amsmath,amsfonts,amsthm,amssymb,color,tikz}
\usepackage[T1]{fontenc}
\usepackage{enumerate}
\usepackage{hyperref}
\hypersetup{
     colorlinks   = true,
     citecolor    = blue,
     linkcolor    = blue
}

\usepackage{pdfsync}
\usepackage[font={scriptsize}]{caption}

\usepackage[left=1in, right=1in, top=1.1in,bottom=1.1in]{geometry}
\newcommand{\der}{\delta}

\newcommand{\hep}{\hat{\varepsilon}}

\newcommand{\iot}{\int_{0}^{t}}

\newcommand{\ist}{\int_{s}^{t}}
\newcommand{\js}{j^{*}}

\newcommand{\1}{{\bf 1}}

\newcommand{\R}{\mathbb R}
\newcommand{\N}{\mathbb N}

\newcommand{\be}{\mathbf{E}}

\newcommand{\cac}{\mathcal C}

\newcommand{\cn}{\mathcal N}

\newcommand{\al}{\alpha}

\newcommand{\ep}{\varepsilon}

\newcommand{\ga}{\gamma}
\newcommand{\gga}{\Gamma}
\newcommand{\ka}{\kappa}
\newcommand{\la}{\lambda}

\newcommand{\si}{\sigma}

\newcommand{\vp}{\varphi}

\newcommand{\lp}{\left(}
\newcommand{\rp}{\right)}
\newcommand{\lc}{\left[}
\newcommand{\rc}{\right]}
\newcommand{\lcl}{\left\{}
\newcommand{\rcl}{\right\}}
\newcommand{\lln}{\left|}
\newcommand{\rrn}{\right|}

\newtheorem{theorem}{Theorem}[section]

\newtheorem{corollary}[theorem]{Corollary}

\newtheorem{hypothesis}[theorem]{Hypothesis}
\newtheorem{lemma}[theorem]{Lemma}

\newtheorem{proposition}[theorem]{Proposition}

\theoremstyle{remark}
\newtheorem{remark}[theorem]{Remark}

\theoremstyle{remark}
\newtheorem{example}[theorem]{Example}

\newcommand{\bean}{\begin{eqnarray*}}
\newcommand{\eean}{\end{eqnarray*}}
\newcommand{\ben}{\begin{enumerate}}
\newcommand{\een}{\end{enumerate}}
\newcommand{\beq}{\begin{equation}}
\newcommand{\eeq}{\end{equation}}

\begin{document}

\title[SDEs with power nonlinearities]{Young differential equations\\ with power type nonlinearities}

\date{\today}

\author[J.A Le\'on \and D. Nualart \and S. Tindel]{Jorge A. Le\'on \and David Nualart \and Samy Tindel}

\address{
Jorge A. Le\'{o}n:
Depto. de Control Autom\'{a}tico, CINVESTAV-IPN, Apartado Postal 14-740, 07000 M\'{e}xico, D.F., Mexico}
\email{jleon@ctrl.cinvestav.mx}

\address{David Nualart: Department of Mathematics, University of Kansas, 405 Snow Hall, Law\-rence, Kansas, USA.}
\email{nualart@ku.edu}

\address{Samy Tindel: Department of Mathematics,
Purdue University,
150 N. University Street,
W. Lafayette, IN 47907,
USA.}
\email{stindel@purdue.edu}

\thanks{J.A. Le\'on was supported by the CONACyT grant 220303}

\thanks{D. Nualart was supported by the NSF grant  DMS1208625 and the ARO grant FED0070445}

\begin{abstract}
In this note we give several methods to construct nontrivial solutions to the equation $dy_{t}=\si(y_{t}) \, dx_{t}$, where $x$ is a $\ga$-H\"older $\R^{d}$-valued signal with $\ga\in(1/2,1)$ and $\si$ is a function behaving like  a power function $|\xi|^{\ka}$, with $\ka\in(0,1)$. In this situation, classical Young integration techniques allow to get existence and uniqueness results whenever $\ga(\ka+1)>1$, while we focus on cases where $\ga(\ka+1)\le 1$. Our analysis then relies on some extensions of Young's integral allowing to cover the situation at hand.
\end{abstract}

\maketitle

\section{Introduction}
\label{sec:intro}

Let $T>0$ be a fixed arbitrary horizon, and consider a noisy function $x:[0,T]\to\R^{d}$  in the H\"older space $\cac^{\ga}([0,T];\, \R^{d})$, with $\ga>1/2$. Let $\si^{1},\ldots,\si^{d}$ be some vector fields on $\R^{m}$, $a$ be an initial data in $\R^{m}$ and consider the following integral equation
\begin{equation}\label{eq:sde-power}
y_{t} = a +  \sum_{j=1}^{d}\int_{0}^{t} \si^{j}(y_{u}) \, dx_{u}^{j},\quad t\in[0,T].
\end{equation}
When $\si^{1},\ldots,\si^{d}$ are smooth enough, equation \eqref{eq:sde-power} can be solved thanks to fractional calculus~\cite{Ly, NR,Za} or Young integration techniques. Extensions of these methods, thanks to the rough paths theory (see e.g \cite{FV-bk,Ly2}), also allow to handle cases of signals with regularity lower than $1/2$.

In the current paper, we are concerned with a different, though very natural problem: can we define and solve equation \eqref{eq:sde-power} for coefficients which are only H\"older continuous? Stated in such a generality the question is still open, but we consider here the special case of a coefficient $\si$ behaving like a power function.

This problem has quite a long story, and a full answer in  the case of a 1-dimensional equation driven by a standard Brownian motion is given in \cite{Lg,WY}. The basic idea on which Watanabe-Yamada's contribution relies, is the following a priori estimate. Consider equation~\eqref{eq:sde-power} driven by a Brownian motion $B$, with a non-linearity $\si(\xi)=|\xi|^{\ka}$ where $\ka>1/2$. Namely, let $y$ be a solution to
\begin{equation}\label{eq:sde-power-brownian}
y_{t} = a + \int_{0}^{t} |y_{u}|^{\ka} \, dB_{u},\quad t\in[0,T],
\end{equation}
where the differential with respect to $B$ is understood in the It\^o sense.
Then obviously the main problem in order to estimate $y$ is its behavior close to 0, since elsewhere $\xi\mapsto |\xi|^{\ka}$ is a Lipschitz function. For $n\ge 1$ we thus consider an approximation $\vp_{n}$ of the function $\xi\mapsto |\xi|$ such that $\vp_{n}\in C_{b}^{2}(\R)$, $\varphi_n \ge 0$  and $\|\vp_{n}^{(2)}\|_{\infty}\le n$. Then applying It\^o's formula to equation~\eqref{eq:sde-power-brownian} we get
\begin{equation}\label{eq:expect-vpn-yt}
\be\lc \vp_{n}(y_{t}) \rc
=
\vp_{n}(a) + \frac12 \int_{0}^{t} \be\lc \vp_{n}^{(2)}(y_{u}) \, |y_{u}|^{2\ka}\rc \, du.
\end{equation}
The right hand side of equation \eqref{eq:expect-vpn-yt} is then controlled by noticing that, whenever $|y_{u}|\le 1/n$, we have $|\vp_{n}^{(2)}(y_{u}) |\, |y_{u}|^{2\ka}\le n^{-(2\ka-1)}$. This quantity converges to 0 as $n\to\infty$, which is the key step in order to control $\be[ \vp_{n}(y_{t}) ]$ in \cite{WY}.

The method described above in order to handle the Brownian case is short and elegant, but fails to give a true intuition of the phenomenon allowing to solve equation \eqref{eq:sde-power} with a power type coefficient. This intuition has been highlighted in \cite{MP,MPS}, though in the much more technical context of the stochastic heat equation. In order to understand the main idea, let us go back to equation \eqref{eq:sde-power} understood in the Young sense. Then two cases can be thought of (we restrict our considerations to 1-dimensional paths in the remainder of the introduction for notational sake):

\noindent
\textit{(i)}
One expects $y$ to be an element of $\cac^{\ga}$, since the equation is driven by $x\in\cac^{\ga}$. This means that $\si(y)$ should lye in $\cac^{\ka\ga}$.
When $\ka$ satisfies $\ka\,\ga+\ga>1$, each integral $\int_{0}^{t} \si(y_{u}) \, dx_{u}$ can thus be defined as a usual Young integral, and equation \eqref{eq:sde-power} is solved thanks to classical methods as in \cite{FV-bk,Ly,Za}.

\noindent
\textit{(ii)}
Let us now consider the case $\ka\,\ga+\ga\le1$.
If one wishes to define the integral $\int_{0}^{t} \si(y_{u}) \, dx_{u}$ properly when $y_{u}$ is close to 0, the heuristic argument is as follows: when $y_{u}$ is small the equation is basically noiseless, so that $\si(y)$ should be considered as a $\cac^{\ka}$-H\"older function  instead of a $\cac^{\ka\ga}$-H\"older function. This means that the expected condition on $\ka$ in order to solve equation~\eqref{eq:sde-power} is just $\ka+\ga>1$.

\noindent
As mentioned above, this strategy has been successfully implemented in \cite{MP,MPS} in a Brownian SPDE context. It  heavily relies  on the regularity gain when $y$ hits 0. In our case, we will follow two directions which are somehow different in their nature: (i) We will see that if $y$ does not hit 0 too sharply, this condition being quantified in an integral way, then the integrals $\int_{0}^{t} \si(y_{u}) \, dx_{u}$ still have a good chance to be defined even if $\ka\,\ga+\ga <1$. One can then construct a  solution of \eqref{eq:sde-power} in this landmark. (ii) Another approach consists in quantifying the regularity gain enforced by equation \eqref{eq:sde-power} when the solution $y$ approaches 0. In this way, one can get some uniform a priori H\"older bounds on  $y$ and invoke some compactness arguments.

To be more specific, we shall proceed as follows:

\noindent
\textbf{(1)}
We start with a general lemma on Young integration. Namely (see Proposition \ref{prop:integral} for a precise statement), we consider $\eta$ such that $(\ka+\eta)\ga>1-\ga$. We also consider a path $y\in\cac^{\ga}$ and a function $\si$ behaving like a power function $|\xi|^{\ka}$. By adding the assumption $|y|^{-1}\in L^{q}([0,\tau])$ with $q=\frac{\eta}{\ga(\ka+\eta)}$, we prove that 
$\int_{0}^{t} \si(y_{u}) \, dx_{u}$ is well defined as a Young-type integral and gives raise to a $\ga$-H\"older function. Notice that we have carried out this part of our program with fractional integration techniques because the calculations are easily expressed in this setting. We can however link the integral we obtain with Riemann sums.

\noindent
\textbf{(2)}
With this integration result in hand, we consider the 1-dimensional version of equation~\eqref{eq:sde-power} and perform a Lamperti-type transformation $y_{t}=\phi^{-1}(x_{t})$, where $\phi(\xi)=\int_{0}^{\xi}[\si(s)]^{-1} ds$. Then we prove that $y$ is a solution to our equation of interest by identifying the Young integral $\int_{0}^{t} \si(y_{u}) \, dx_{u}$ for $y_{t}=\phi^{-1}(x_{t})$. Our result is valid for any $\ka$ such that $\ga(1+\ka)<1$, and we refer to Theorem \ref{th1v1} for a precise statement.

\noindent
\textbf{(3)}
Our basic a priori estimate for \eqref{eq:sde-power} states that whenever a solution $y$ satisfies $|y_{u}|\le 2^{-k}$ for $u$ lying in an interval $I$, then we also have $|y_{t}-y_{s}|$ of order $2^{-\ka k}|t-s|^{\ga}$ for $s,t\in I$. Our regularity gain is thus expressed by the coefficient $2^{-\ka k}$ above. This gain is sufficient to get  to the existence of a $\gamma$-H\"older continuous  solution to equation \eqref{eq:sde-power} in the $d$-dimensional case. This solution vanishes as soon as it hits the origin (see Theorem \ref{thm:exist-multidim}).

\noindent
Summarizing the considerations above, we are able to get existence theorems for equation~\eqref{eq:sde-power} with power type nonlinearities in a wide range of cases. The situation would obviously be clearer if we could get the corresponding pathwise uniqueness results, like in the aforementioned references \cite{Lg,MP,MPS,WY}. However, these articles handle the case of It\^o type equations, for which uniqueness is expected. In our Stratonovich-Young case uniqueness of the solution is ruled out, since both the nontrivial solution we shall construct and the solution $y\equiv 0$ solve equation \eqref{eq:sde-power} when $a=0$. We shall go back to this issue below.

Our paper is structured as follows: an extension of Young's integral related to our power type coefficient is constructed in Section 2. Section 3 deals with its application to the existence of solutions to equation \eqref{eq:sde-power} in dimension 1. The other approach, based on the a priori regularity gain of the solution when it hits 0, is developed in Section \ref{sec:increment-control}. Finally, in Section 5 we discuss the application of these results to the case of stochastic differential equations driven by a  fractional Brownian motion.

\paragraph{\textbf{Notations:}} Throughout the article, we use the following conventions: for 2 quantities $a$ and $b$, we write $a \lesssim b$ if there exists a universal constant $c$ (which might depend on the parameters of the model, such as, $\ga$, $\ka$, $\eta$, $\alpha$, $T$,...) such that $a\le c\,b$.  If $f$ is a vector-valued function defined on an interval $[0,T]$ and $s,t\in [0,T]$,  $\der f_{st}$ denotes  the increment $f_{t}-f_{s}$.

\section{An extension of Young's integral}\label{sec:frac-calc-approach}

This section is devoted to an extension of  Young's integral using  fractional calculus techniques, which will be suitable to handle  equation \eqref{eq:sde-power} with H\"older-type and singular nonlinearities. We shall first recall some general elements of fractional calculus.

\subsection{Elements of fractional calculus}\label{sec:frac-calc}
We restrict this introduction to real-valued functions for notational sake.
Consider $0\le a < b \le  T$ and an $L^{1}([0,T])$-function $f$. For $t\in [a,b]$ and $\al\in(0,1)$ the fractional integrals of $f$ are defined  as
\begin{equation*}
I_{a+}^{\al} f_{t} = \frac{1}{\gga(\al)} \int_{a}^{t} (t-r)^{\al-1} f_{r} \, dr,
\quad\mbox{and}\quad
I_{b-}^{\al} f_{t} = \frac{1}{\gga(\al)} \int_{t}^{b} (r-t)^{\al-1} f_{r} \, dr .
\end{equation*}
For any $p\ge 1$, we denote by $I_{a+}^{\al}(L^{p})$ the image of $L^{p}([a,b])$ by $I_{a+}^{\al}$, and similarly for $I_{b-}^{\al}(L^{p})$.

The inverse of the operators $I_{a+}^{\al}$ and $I_{b-}^{\al}$ are called fractional derivatives, and are defined as follows. For $f\in I_{a+}^{\al}(L^{p})$ 
and $t\in[a,b]$ we set
\begin{equation}\label{def:der+}
D_{a+}^{\al} f_{t}
=L^p-\lim_{\varepsilon\downarrow0}\frac{1}{\Gamma(1-\alpha)}
\lp
\frac{f_{t}}{(t-a)^{\al}}
+ \al \int_{a}^{t-\varepsilon} \frac{f_{t}-f_{r}}{(t-r)^{1+\al}} \, dr
\rp,
\end{equation}
where we use the convention $f_r=0$ on $[a,b]^c$.
In the same way, for $f\in I_{b-}^{\al}(L^{p})$  and $t\in[a,b]$, we set
\begin{equation}\label{def:der-}
D_{b-}^{\al} f_{t}
=L^p-\lim_{\varepsilon\downarrow0}\frac{1}{\Gamma(1-\alpha)}
\lp
\frac{f_{t}}{(b-t)^{\al}}
+ \al \int_{t+\varepsilon}^{b} \frac{f_{t}-f_{r}}{(r-t)^{1+\al}} \, dr
\rp.
\end{equation}
By \cite[Remark 13.2]{Sam} we have  that, for $p> 1$, $f\in I_{a+}^{\al}(L^{p})$ (resp.
$f\in I_{b-}^{\al}(L^{p})$) if and only if $f\in L^{p}([a,b])$ and the limit in the
right-hand side of (\ref{def:der+}) (resp. (\ref{def:der-})) exists. In this case
$f=I^{\alpha}_{a+}(D_{a+}^{\al} f)$ (resp. $f=I^{\alpha}_{b-}(D_{b-}^{\al} f)$).
It is not difficult to see that, as a consequence of the proof of
\cite[Theorem 13.2]{Sam}, the fact that $f\in L^{p}([a,b])$, $\frac{f(\cdot)}
{(\cdot-a)^{\alpha}}$ and $\int_a^{\cdot}\frac{f(\cdot)-f_r}{(\cdot-r)^{1+\alpha}}dr$
(resp. $\frac{f(\cdot)}
{(b-\cdot)^{\alpha}}$ and $\int_{\cdot}^b\frac{f(\cdot)-f_r}{(r-\cdot)^{1+\alpha}}dr$)
belong to $L^p([a,b])$ implies that $f\in I_{a+}^{\al}(L^{p})$ (resp.
$f\in I_{b-}^{\al}(L^{p})$) and 
\begin{equation}\label{eq:def-Da-plus}
D_{a+}^{\al} f_{t}
=\frac{1}{\Gamma(1-\alpha)}
\lp
\frac{f_{t}}{(t-a)^{\al}}
+ \al \int_{a}^{t} \frac{f_{t}-f_{r}}{(t-r)^{1+\al}} \, dr
\rp
\end{equation}
(resp.
$$D_{b-}^{\al} f_{t}
=\frac{1}{\Gamma(1-\alpha)}
\lp
\frac{f_{t}}{(b-t)^{\al}}
+ \al \int_{t}^{b} \frac{f_{t}-f_{r}}{(r-t)^{1+\al}} \, dr
\rp).
$$
Notice that $\cac^{\al+\varepsilon}([a,b])\subset I_{a+}^{\al}(L^{p})$, with
$\varepsilon>0$. In the same manner, we have $\cac^{\al+\varepsilon}([a,b])\subset I_{b-}^{\al}(L^{p})$.

Let $g, f\in L^1([0,T])$ be two functions
such that, for some $\alpha\in(0,1)$, $f\in I_{a+}^{\al}(L^{1})$ and
$g^{b-}\in I_{b-}^{1-\al}(L^{1})$, where $g_r^{b-}=g_r-g_{b-}$. In this case we say that
$f$ is integrable with respect to $g$ if and only if 
$(D_{a+}^{\al} f)  D_{b-}^{1-\al}g_{r}^{b-}\in L^{1}([a,b])$. In this case we 
define the integral $\int_a^b f\, dg$ in the following way
\begin{equation}\label{eq:def-intg-frac}
\int_{a}^{b} f_{r}\, dg_{r}
:=
\int_{a}^{b} (D_{a+}^{\al} f_{r}) D_{b-}^{1-\al}g_{r}^{b-} \, dr.
\end{equation}
 Under our assumptions, it can be checked that $\int_{a}^{b} f_{r}\, dg_{r}$ is well-defined, and that it coincides with Young's integral defined as a limit of Riemann sums. We shall derive below an extension of this integral suited to our purposes.

\subsection{An extension of the fractional integral}
We assume in this section that  $x$ is real valued.
Consider the following additional assumption on the coefficient $\si: \R^m \rightarrow \R^m$.
\begin{hypothesis}\label{hyp:sigma}
The function $\si:\R^m\to\R^m$ satisfies $\si(0)=0$ and
\begin{equation}\label{eq:dsigma}
|\si(\xi_{2})-\si(\xi_{1})| \lesssim \left||\xi_{2}|^{\ka}-|\xi_{1}|^{\ka}\right|,
\quad \xi_1,\xi_2\in\R^m,
\end{equation}
for some $\ka\in(0,1)$ such that $\ga(\ka+1)<1$.
\end{hypothesis}
\begin{remark}\label{re:2.2}
In order to understand the implications of Hypothesis \ref{hyp:sigma}, note that if $\si$ fulfills condition \eqref{eq:dsigma} and if we consider $\xi_{1},\xi_{2}\in\R^{m}$ such that $|\xi_{1}|=|\xi_{2}|$, then we obviously have $\si(\xi_{2})=\si(\xi_{1})$. Thus (\ref{eq:dsigma})  implies that $\sigma$ is a radial function, that is, $\sigma(\xi)=\rho(|\xi|)$, where $\rho: [0,\infty) \to \R^m$.
 On the other hand, it is not difficult to see that  a radial function $\sigma(\xi)=\rho(|\xi|)$  such that $\rho \in \mathcal{C}^1((0,\infty))$, $\rho(0)=0$ and
$|\rho^{(1)}(y)|\lesssim  y^{\ka-1}$, $y>0$, satisfies inequality (\ref{eq:dsigma}).
\end{remark}

For a function $\sigma$ satisfying Hypothesis 2.1, we  define   
\begin{equation}\label{eq:def-holder-semi-norm}
\cn_{\ka,\si}
:= 
\sup\lcl \frac{\lln\sigma(\xi_2)-\sigma(\xi_1)\rrn}{\left| |\xi_2|^{\ka}-|\xi_1|^{\ka}\right|} 
:\, \xi_2,\xi_1\in\mathbb{R}^m, |\xi_{1}|\ne |\xi_{2}|\rcl.
\end{equation}

We now label the following  auxiliary result for further use.

\begin{lemma}\label{lem:psigma}
Assume $\si$ satisfies Hypothesis \ref{hyp:sigma}. Then we have
$$|\si(\xi_{2})-\si(\xi_{1})| \le \frac{ \ka}{ \ka+\eta}    \cn_{\ka,\si} \left(|\xi_2|^{-\eta}+|\xi_1|^{-\eta}\right)
\left|\xi_2-\xi_1\right|^{\ka+\eta},
$$
for any $0\le \eta \le 1-\ka$ and $\xi_1,\xi_2\in\R^m\setminus\{ 0\}$.
\end{lemma}
\begin{proof}  The case $\eta =0$ or $\eta= 1-\ka$ is obvious, so we assume $0< \eta < 1-\ka$.
Without loss of generality, we can assume that $|\xi_1|\le |\xi_2|$.
According to (\ref{eq:dsigma}), we can write
\begin{eqnarray*}
|\si(\xi_{2})-\si(\xi_{1})|   
&\le &  \cn_{\ka,\si} \left( |\xi_{2}|^{\ka}-|\xi_{1}|^{\ka}  \right)
= \ka \cn_{\ka,\si}\int_{|\xi_1|}^{|\xi_2|}z^{\ka-1}dz
\le\ka  \cn_{\ka,\si}|\xi_1|^{-\eta}
\int_{|\xi_1|}^{|\xi_2|}z^{\ka+\eta-1}dz\\
&\le &  \ka    \cn_{\ka,\si} |\xi_1|^{-\eta}
\int_{|\xi_1|}^{|\xi_2|}\left(z-|\xi_1|\right)^{\ka+\eta-1}dz,
\end{eqnarray*}
which yields our claim.
\end{proof}

We are now ready to establish our extension result for the Young's integral.

\begin{proposition}\label{prop:integral}
Assume that $\sigma $ satisfies Hypothesis \ref{hyp:sigma}, and recall that we consider $\ka,\ga$ such that $1-\ga(\ka+1)>0$.  For any $\gamma\in (0,1)$ and $\eta>0$ we introduce the space
\begin{equation}\label{eq:def-C-gamma-eta}
\mathcal{C}^\ga_\eta ([0,T] ;\R^m) =\{y \in \mathcal{C}^\ga  ([0,T] ;\R^m): |y|^{-1} \in L^{\eta}([0,T];\R)\}.
\end{equation}
Then the following results hold true:

\noindent
\emph{(i)}
If  $y\in\cac^{\ga}_\eta([0,T];\R^m) $   for some $\eta$ such that $\frac{1-\ga(1+\ka)}{\ga}<\eta<1-\ka$, then, for any $t\in [0,T]$, the integral
\begin{equation*}
\lc\Lambda(y)\rc_{t}: = \int_0^{t}\sigma(y_s) \, dx_s,
\end{equation*}
is well defined  in the sense of relation~\eqref{eq:def-intg-frac}. 

\noindent
\emph{(ii)}
Notice that if $\eta<1-\ka$, then we have $\ga(\ka+\eta)<\ga<1$.
Now if $y $  satisfies the stronger condition  $y \in \mathcal{C}^\ga_{ \frac  \eta  {\ga(\ka+\eta)}} ([0,T] ;\R^m) $, then $\Lambda(y)$ belongs to  the space $\cac^{\ga}([0,T];\R^m)$, and
\begin{equation}\label{eq:bnd-Gamma-y-with-intg}
\left\|\Lambda(y)\right\|_{\ga}
\lesssim
\|x\|_{\ga}\left(\|\sigma(y)\|_{\infty}+\cn_{\ka,\si} \|y\|_{\ga}^{\ka+\eta}
\left(\int_{0}^{T}|y_s|^{-\frac \eta { \ga(\ka+\eta)}}ds\right)^{\ga(\ka+\eta)}\right),
\end{equation}
where $\cn_{\ka,\si}$ has been introduced in \eqref{eq:def-holder-semi-norm}.
\end{proposition}

\begin{remark} 
Taking into account that the function $\eta \to \frac \eta{ \eta+\ka}$ is strictly increasing we deduce that
  $\eta >\frac  1 \ga  -1-\ka  $ if and only if  $\frac \eta { \ga(\ka+\eta)} >\frac {1-\ga-\ka\ga}{\ga (1-\ga)}$. Therefore, in condition (ii) the integrability condition for $|y|^{-1}$ is of order strictly larger than 
$\frac {1-\ga-\ka\ga}{\ga (1-\ga)}$.
\end{remark}

\begin{proof}[Proof of Proposition \ref{prop:integral}] 
Let $\alpha$  be such that $1-\ga<\alpha< \ga (\ka +\eta)$, which implies   $\alpha\ga^{-1}-\ka<\eta<1-\ka$. 
Let $0\le t_1<t_2\le T$.  Recall that the integral
$\int_{t_{1}}^{t_{2}} [\si(y)]_s\, dx_{s}$ is defined by formula~\eqref{eq:def-intg-frac}. To show that this integral exists and to establish suitable estimates,  we first analyze the fractional derivative of $x$
\begin{eqnarray}
\left|  D^{1-\alpha}_{t_2-}x^{t_2-}_s \right|
&=&
\frac{1}{\Gamma(\alpha)}
\left|\frac{x_s-x_{t_2}}{(t_2-s)^{1-\alpha}}
+(1-\alpha)\int_s^{t_2} \frac{x_s-x_{r}}{(r-s)^{2-\alpha}}dr\right|\nonumber\\
&\lesssim& \|x\|_{\ga}(t_2-s)^{\alpha+\ga-1}+
\|x\|_{\ga}\int_s^{t_2}(r-s)^{\alpha+\ga-2}dr\nonumber\\
&\lesssim& \|x\|_{\ga}(t_2-s)^{\alpha+\ga-1},\label{eq:dfracx}
\end{eqnarray}
where we have used the fact that $\al+\ga>1$ for the last step. Hence, we can write
\begin{equation*}
  \int_{t_{1}}^{t_{2}}  \left| [D_{t_{1}+}^{\al} \si(y)]_{s} \, D_{t_{2}-}^{1-\al}x_{s}^{t_{2}-}  \right| ds  
\lesssim  \| x\| _\ga \left(J_{t_{1}t_{2}}^{1} + J_{t_{1}t_{2}}^{2}\right),
\end{equation*}
with
\[
J_{t_{1}t_{2}}^{1}=
\|\sigma(y)\|_{\infty}\int_{t_1}^{t_2}(s-t_1)^{-\alpha}(t_2-s)^{\alpha+\ga-1}ds 
\]
and
\[
J_{t_{1}t_{2}}^{2}=
\int_{t_1}^{t_2}\left(\int_{t_1}^{s}\frac{|\sigma(y_s)-\sigma(y_u)|}{(s-u)^{\alpha+1}}
du\right)(t_2-s)^{\alpha+\ga-1}ds.
\]
It is now readily checked that 
\begin{equation}  \label{eq1}
J_{t_{1}t_{2}}^{1}\lesssim \|\sigma(y)\|_{\infty}  (t_{2}-t_{1})^{\ga}.
\end{equation}
For the term $J_{t_{1}t_{2}}^{2}$, invoking Lemma \ref{lem:psigma} and some elementary algebraic manipulations, we get
\begin{eqnarray} \label{eq:inty}
J_{t_{1}t_{2}}^{2}&\lesssim&   \cn_{\kappa, \sigma}
\int_{t_1}^{t_2}(t_2-s)^{\alpha+\ga-1}
\int_{t_1}^{s}
\left(|y_s|^{-\eta}+|y_u|^{-\eta}\right)\frac{|y_s-y_u|^{\ka+\eta}}{(s-u)^{\alpha+1}}duds \notag\\
&\lesssim& \cn_{\kappa, \sigma} \|y\|_{\ga}^{\ka+\eta}\int_{t_1}^{t_2}(t_2-s)^{\alpha+\ga-1}
\int_{t_1}^{s}
\left(|y_s|^{-\eta}+|y_u|^{-\eta}\right)(s-u)^{\ga(\ka+\eta)-\alpha-1}duds \notag\\
&\lesssim& \cn_{\kappa, \sigma} \|y\|_{\ga}^{\ka+\eta}  \Bigg( \int_{t_1}^{t_2}(t_2-s)^{\alpha+\ga-1}|y_s|^{-\eta}\int_{t_1}^{s}
(s-u)^{\ga(\ka+\eta)-\alpha-1}duds  \\
&&
\hspace{2in}+\int_{t_1}^{t_2}|y_u|^{-\eta}\int_{u}^{t_2}(t_2-s)^{\alpha+\ga-1}(s-u)^{\ga(\ka+\eta)-\alpha-1}dsdu\Bigg) . \notag
\end{eqnarray}
Notice that   $\eta>\alpha\ga^{-1}-\ka$ implies that
$\ga(\ka+\eta)-\alpha>0$.  This implies that the integral  $\int_{t_{1}}^{t_{2}} [\si(y)]_{s}\, dx_{s}$  is well defined, provided   $|y|^{-1} \in L^\eta([0,T];\R)$.

 Applying H\"older's inequality with $p^{-1}=\ga(\ka+\eta)$ and $q^{-1}=1-p^{-1}$,  and assuming  $ |y|^{-1} \in L^{\eta/(\ga(\ka+\eta))}([0,T];\R) $,  yields
\begin{multline*}
J_{t_{1}t_{2}}^{2}
\lesssim  \cn_{\kappa, \sigma} \|y\|_{\ga}^{\ka+\eta} 
\left(\int_{t_1}^{t_2}|y_u|^{-p\eta}du\right)^{1/p}
\Bigg[
\left(\int_{t_1}^{t_2}(t_2-s)^{q(\alpha+\ga-1)}
(s-t_1)^{q(\ga(\ka+\eta)-\alpha)}ds\right)^{1/q} \\
+
\left(\int_{t_1}^{t_2}(t_2-u)^{q\ga(\ka+\eta+1)-q}du\right)^{1/q}
 \Bigg].
\end{multline*}
Now a simple analysis of the exponents in the above relation  implies 
\begin{equation} \label{eq2} 
J_{t_{1}t_{2}}^{2}\lesssim   \cn_{\kappa, \sigma} \|y\|_{\ga}^{\ka+\eta} \left(\int_{0}^{T}|y_s|^{-\frac{\eta}{\ga(\ka+\eta)}}ds\right)^{\ga(\ka+\eta)}  (t_2-t_1)^{\ga}.
\end{equation}
Finally, the estimates \eqref{eq:bnd-Gamma-y-with-intg} follows from  \eqref{eq1} and \eqref{eq2}. The proof is now complete.
\end{proof}

\subsection{The integral via Riemann sums}
The next goal is to see that the integral $\Lambda(y)$ given in Proposition \ref{prop:integral} can be approximated by Riemann sums. Towards this end,  for any $n\ge 2$, we consider
a  uniform partition $\Pi_n=\{a=t_1<t_2<\dots<t_n=b\}$ of the interval $[a,b]\subset[0,T]$,
such that $ \left| \Pi_n \right|:=\frac  {b-a}{n-1} =t_{j+1}-t_j   $ for all $j \in\{1,2,\ldots,n-1\}$. For $y$ as in Proposition \ref{prop:integral} (i), we define the following approximation based on $\Pi_{n}$
\begin{equation}\label{eq:def-y-n}
z^{n}_s=\sum_{i=2}^n\frac{1}{\left|\Pi_n\right|}\left(\int_{t_{i-1}}^{t_i}\sigma(y_s)
ds\right)\1_{(t_{i-1},t_i]}(s),\quad  s\in[a,b].
\end{equation}
We observe that, owing  to \cite[Corollary 2.3]{Za}, we have
$$\int_a^{b} z_{s}^{n}dx_s=\sum_{i=2}^n\frac{1}{\left|\Pi_n\right|}\left(\int_{t_{i-1}}^{t_i}\sigma(y_s)
ds\right) \delta x_{t_{i-1}t_{i}},
$$
where the left hand side is understood as in relation \eqref{eq:def-intg-frac} and where we recall that $\delta x_{uv}:= x_{v} - x_{u}$.   
The convergence of $\int_a^b z^{n}_s \, dx_s$ is given in the following theorem, which is the main result of this subsection.

\begin{theorem}\label{thm:lim-riemann-sums}
Suppose that $\sigma$ satisfies Hypothesis \ref{hyp:sigma}. 
Let   $\eta$ be such that $\frac{1-\ga(1+\ka)}{\ga}<\eta<1-\ka$. Consider
$y \in \cac_{\eta}^{\ga}([0,T];\R^m)$ as introduced in \eqref{eq:def-C-gamma-eta}, and 
recall that $z^{n}$ is defined by~\eqref{eq:def-y-n}. Then for all $0 \le a<b \le T$ we have
$$
\lim_{n\to\infty}\int_a^b z_{s}^{n}dx_s
= 
\int_a^b\sigma(y_s)dx_s.
$$
\end{theorem}

 In order to prove this theorem, we first go through a series of  auxiliary results.

\begin{lemma}\label{lem:aux2}
Let $\sigma$ satisfy Hypothesis \ref{hyp:sigma}, $y\in \cac^{\ga}([0,T];\R^m)$ and consider $[a,b]\subset[0,T]$. Then for all $s\in [a,b]$ we have
$$\left|\sigma(y_s)-z_{s}^{n}\right|\le
 \cn_{\ka,\si} \|y\|_{\ga}^{\ka}\left|\Pi_n\right|^{\ka\ga}. 
$$
\end{lemma}
\begin{proof}
For $s\in (a,b]$, the definition of $z^{n}$ gives
\begin{eqnarray*}
\left|\sigma(y_s)-z_{s}^{n}\right|&=&\sum_{i=2}^n\left|\sigma(y_s)-
\frac{1}{\left|\Pi_n\right|}\int_{t_{i-1}}^{t_i}\sigma(y_r)
dr\right|\1_{(t_{i-1},t_i]}(s)\\
&\le&\sum_{i=2}^n\frac{1}{\left|\Pi_n\right|}\left(\int_{t_{i-1}}^{t_i}|\sigma(y_s)-\sigma(y_r)|
dr\right)\1_{(t_{i-1},t_i]}(s)\\
&\le& \cn_{\ka,\si} \sum_{i=2}^n\frac{1}{\left|\Pi_n\right|}
\left(\int_{t_{i-1}}^{t_i} \big| \, |y_s|^{\ka}-|y_r|^{\ka} \big|
dr\right)\1_{(t_{i-1},t_i]}(s).
\end{eqnarray*}
Since $y$ is $\ga$-H\"older continuous, we thus have
\begin{eqnarray*}
\left|\sigma(y_s)-z_{s}^{n}\right|&\le&\cn_{\ka,\si}\|y\|_{\ga}^{\ka}
\sum_{i=2}^n\frac{1}{\left|\Pi_n\right|}\left(\int_{t_{i-1}}^{t_i}|s
-r|^{\ka\ga}
dr\right)\1_{(t_{i-1},t_i]}(s)\\
&\le&\cn_{\ka,\si} \|y\|_{\ga}^{\ka}
\sum_{i=2}^n|\Pi_n|^{\ka\ga}\1_{(t_{i-1},t_i]}(s),
\end{eqnarray*}
which completes the proof.
\end{proof}

We now estimate the H\"older regularity of our approximation $z^{n}$.  
\begin{lemma} \label{lem:4.7}
 Let $\sigma$ and $y$ be functions verifying the assumptions of Theorem \ref{thm:lim-riemann-sums}. Then, 
for $a< u<s\le b$, we have
\begin{equation*}
\big|z_{s}^{n}-z_{u}^{n}\big|
\lesssim \|y\|_{\ga}^{\ka+\eta} \lp \Phi^{n}_{u,s} + \Psi^{n}_{u,s}  \rp,
\end{equation*}
where 
\[
\Phi^{n}_{u,s}=
|\Pi_n|^{\ga(\ka+\eta)-1}
\sum_{2\le j<i\le n}
\left(\int_{t_{i-1}}^{t_i}|y_{r}|^{-\eta}dr +\int_{t_{j-1}}^{t_j}|y_{r}|^{-\eta}dr\right) 
\1_{(t_{j-1},t_j]}(u)\1_{(t_{i-1},t_i]}(s)
\]
and
\[
\Psi^{n}_{u,s}=
\frac{(s-u)^{\ga(\ka+\eta)}}{|\Pi_n|}
\sum_{2\le j<i\le n}
\left(\int_{t_{i-1}}^{t_i}|y_{r}|^{-\eta}dr +\int_{t_{j-1}}^{t_j}|y_{r}|^{-\eta}dr\right) 
\1_{(t_{j-1},t_j]}(u)\1_{(t_{i-1},t_i]}(s).
\]
\end{lemma}

\begin{proof}
Assume $s\in(t_{i-1},t_{i}]$. If $u$ lies into $(t_{i-1},t_{i}]$ too, then $|z_{s}^{n}-z_{u}^{n}|=0$ by definition of $z^{n}$. We now assume that $u\in(t_{j-1},t_{j}]$ with $j\in\{2,\ldots,i-1\}$. Then it is readily checked that
\begin{eqnarray*}
z_{s}^{n}-z_{u}^{n}
&=&
\frac{1}{|\Pi_{n}|} \lp \int_{t_{i-1}}^{t_{i}} \si(y_{r}) \, dr  - \int_{t_{j-1}}^{t_{j}} \si(y_{r}) \, dr\rp \\
&=&
\frac{1}{|\Pi_{n}|} \int_{t_{j-1}}^{t_{j}} \lp \si(y_{r+t_{i-1}-t_{j-1}}) - \si(y_{r}) \rp \, dr.
\end{eqnarray*}
Therefore, thanks to Lemma \ref{lem:psigma} we obtain
\begin{eqnarray*}
\big|z_{s}^{n}-z_{u}^{n}\big|
&\lesssim&\cn_{\ka,\si}
\frac{1}{|\Pi_{n}|} \int_{t_{j-1}}^{t_{j}} 
\lp |y_{r+t_{i-1}-t_{j-1}}|^{-\eta} + |y_{r}|^{-\eta} \rp
\lln y_{r+t_{i-1}-t_{j-1}} - y_{r} \rrn^{\ka+\eta} \, dr \\
&\lesssim&\cn_{\ka,\si}
\frac{\|y\|_{\ga}^{\ka+\eta} \, |t_{i-1}-t_{j-1}|^{\ga(\ka+\eta)}}{|\Pi_{n}|} \int_{t_{j-1}}^{t_{j}} 
\lp |y_{r+t_{i-1}-t_{j-1}}|^{-\eta} + |y_{r}|^{-\eta} \rp \, dr,
\end{eqnarray*}
from which we derive
\begin{multline*}
\big|z_{s}^{n}-z_{u}^{n}\big|
\lesssim\frac{\|y\|_{\ga}^{\ka+\eta}(s-u+|\Pi_{n}|)^{\ga(\ka+\eta)}}{|\Pi_n|} \\
\times\sum_{2\le j<i\le n}
\left(\int_{t_{i-1}}^{t_i}|y_{r}|^{-\eta}dr +\int_{t_{j-1}}^{t_j}|y_{r}|^{-\eta}dr\right) \1_{(t_{j-1},t_j]}(u)\1_{(t_{i-1},t_i]}(s).
\end{multline*}
Our claim is now easily deduced.
\end{proof}

The next result will help to handle some of the terms appearing in Lemma \ref{lem:4.7}.
\begin{lemma}\label{lem:4.8}
Let the assumptions of Theorem \ref{thm:lim-riemann-sums} prevail, and consider the path $\Phi^{n}:[a,b]^2\to\R_{+}$ introduced in Lemma \ref{lem:4.7}. We also introduce the following measure on $[a,b]^2$
\begin{equation}\label{eq:def-mu}
\mu(du,ds)=  (s-u) ^{-\al -1} (b-s) ^{\al +\ga -1}\mathbf{1}_{\{u<s\}}  \, duds ,
\end{equation}
where  $\alpha$  is  such that $1-\ga<\alpha< \ga (\ka +\eta)$. Then $\Phi^{n}$ converges to zero in $L^1([a,b]^2, \mu)$, as $n\rightarrow\infty$.
\end{lemma}

\begin{proof}
We can write
\begin{multline}\label{eq:bnd-Phi-L1}
\|\Phi^{n}\|_{L^{1}([a,b]^2,\mu)} \lesssim 
|\Pi_n|^{-1+\ga(\ka+\eta)}
\sum_{2\le j<i\le n}
\left(\int_{t_{i-1}}^{t_i}|y_{r}|^{-\eta}dr 
+\int_{t_{j-1}}^{t_j}|y_{r}|^{-\eta}dr\right) \\
\times\int_{t_{i-1}}^{t_i}\int_{t_{j-1}}^{t_j}(s-u)^{-1-\alpha}duds  
\lesssim 
I_{1}^{n}+I_{2}^{n},
\end{multline}
where
\[
I_{1}^{n}=
|\Pi_n|^{-1+\ga(\ka+\eta)}
\sum_{i=3}^n
\left(\int_{t_{i-1}}^{t_i}|y_{r}|^{-\eta}dr 
+\int_{t_{i-2}}^{t_{i-1}}|y_{r}|^{-\eta}dr\right)
\int_{t_{i-1}}^{t_i}\int_{t_{i-2}}^{t_{i-1}}(s-u)^{-1-\alpha}duds  
\]
and
\[
I_{2}^{n}=
|\Pi_n|^{-1+\ga(\ka+\eta)}
\sum_{i=4}^n\sum_{j=2}^{i-2}
\left(\int_{t_{i-1}}^{t_i}|y_{r}|^{-\eta}dr 
+\int_{t_{j-1}}^{t_j}|y_{r}|^{-\eta}dr\right)
\int_{t_{i-1}}^{t_i}\int_{t_{j-1}}^{t_j}(s-u)^{-1-\alpha}duds.
\]
We now bound the terms $I_{1}^{n}$ and $I_{2}^{n}$ separately.

It is easily seen from the expression of $I_{1}^{n}$ that
\begin{equation*}
I_1^{n}\lesssim |\Pi_n|^{\ga(\ka+\eta)-\alpha}
\sum_{i=2}^n
\int_{t_{i-1}}^{t_i}|y_{r}|^{-\eta}dr =|\Pi_n|^{\ga(\ka+\eta)-\alpha}
\int_{a}^{b}|y_{r}|^{-\eta}dr.
\end{equation*}
Hence, due to the fact that $\ga(\ka+\eta)-\alpha>0$, we obtain $\lim_{n\to\infty}I_{1}^{n}=0$.

As far as $I_{2}^{n}$ is concerned, a simple scaling argument entails
\begin{equation*}
I_2^{n}\lesssim
|\Pi_n|^{\ga(\ka+\eta)-\alpha}
\sum_{i=4}^n\sum_{j=2}^{i-2}
\left(\int_{t_{i-1}}^{t_i}|y_{r}|^{-\eta}dr 
+\int_{t_{j-1}}^{t_j}|y_{r}|^{-\eta}dr\right) 
\int_{i-1}^{i}\int_{j-1}^j(s-u)^{-1-\alpha}duds,
\end{equation*}
and roughly bounding the term $s-u$ by $i-j-1$ in the integral above, we get
\begin{eqnarray*}
I_2^{n}
&\lesssim &|\Pi_n|^{\ga(\ka+\eta)-\alpha}
\sum_{i=4}^n\sum_{j=2}^{i-2}
\left(\int_{t_{i-1}}^{t_i}|y_{r}|^{-\eta}dr 
+\int_{t_{j-1}}^{t_j}|y_{r}|^{-\eta}dr\right) (i-j-1)^{-1-\alpha}\\
&\lesssim &|\Pi_n|^{\ga(\ka+\eta)-\alpha}
\sum_{i=2}^n\int_{t_{i-1}}^{t_i}|y_{r}|^{-\eta}dr
\sum_{k=1}^{n-1}k^{-1-\alpha}
\lesssim |\Pi_n|^{\ga(\ka+\eta)-\alpha}\int_{a}^{b}|y_{r}|^{-\eta}dr.
\end{eqnarray*}
We thus get $\lim_{n\to\infty}I_{2}^{n}=0$, again according to the fact that $\ga(\ka+\eta)-\alpha>0$.

Finally, taking into account $\lim_{n\to\infty}I_{1}^{n}=0$, $\lim_{n\to\infty}I_{2}^{n}=0$ and relation \eqref{eq:bnd-Phi-L1}, our claim is now proved.
\end{proof}

Still having in mind a bound on the terms of  Lemma \ref{lem:4.7}, we state the following intermediate result.
\begin{lemma}\label{lem:a4.9}
 Assume the hypotheses of Lemma \ref{lem:4.8} hold true, and recall that $\Psi^{n}$ is introduced in Lemma \ref{lem:4.7}.
Then as $n\to\infty$, $\Psi^{n}$ converges in $L^1([a,b]^2, \mu)$ to the function $\Psi$ defined as follows
\begin{equation*}
\Psi_{u,s}=
\left(|y_s|^{-\eta}+|y_{u}|^{-\eta}\right) (s-u)^{\ga(\ka+\eta)} \mathbf{1}_{\{u<s\}}.
\end{equation*}
\end{lemma}

\begin{proof}
The result is an immediate consequence of the fact that $|y|^{-\eta}\in L^1([a,b])$, together with the conditions
$\alpha+\ga-1>0$ and $\ga(\ka+\eta)>\al$.
\end{proof}

We are now ready to give the proof of Theorem \ref{thm:lim-riemann-sums}.
 
\begin{proof}[Proof of Theorem \ref{thm:lim-riemann-sums}] Let $\alpha$  be such that $1-\ga<\alpha< \ga (\ka +\eta)$. Owing to \eqref{eq:dfracx} we can write
\[
\lln\int_a^b \lp z_{s}^{n} - \si(y_{s}) \rp \, dx_s \rrn
=
\left|\int_{a}^{b} \left[D^{\alpha}_{a+}\left(\sigma(y)-z^{n}\right)\right]_s
D^{1-\alpha}_{b-}x^{b-}_s ds\right|
\lesssim
 \|x\|_\ga \left( L_{1}^{n}+L_{2}^{n}\right),
\]
where
\[
L_{1}^{n}=  \int_{a}^{b}\frac{\left|\sigma(y_s)-z^{n}_s\right|}{
(s-a)^{\alpha}}(b-s)^{\alpha+\ga-1}ds 
\]
and
\[
L_{2}^{n}=
\int_{a}^{b}\left(\int_{a}^{s}\frac{\Big|\sigma(y_s)
-z^{n}_s-\big(\sigma(y_{u})-z^{n}_u\big)\Big|}{(s-u)^{\alpha+1}}
du\right)(b-s)^{\alpha+\ga-1}ds.
\]
Moreover, notice that invoking Lemma \ref{lem:aux2} we can deduce that $L_{1}^{n}\lesssim
 \cn_{\ka,\si} \|y\|_{\ga}^{\ka} |\Pi_n|^{\ka\ga}$.
Therefore $L_{1}^{n}$ goes to zero as $n\rightarrow\infty$. Thus, in order to finish the
proof we only need to see that $L_{2}^{n}$ converges to zero as $n\rightarrow\infty$.

In order to study the limit of $L_{2}^{n}$, first notice that thanks to Lemma \ref{lem:aux2} we can write 
\begin{equation}  \label{15}
\left|\sigma(y_s)-z^{n}_s-\left(\sigma(y_u)-z^{n}_u\right)\right|
\le
\left|\sigma(y_s)-z^{n}_s\right|+\left|\sigma(y_u)-z^{n}_u\right|
\lesssim \cn_{\ka,\si} \|y\|_{\ga}^{\ka}  |\Pi_{n}|^{\ka\ga},
\end{equation}
which implies that the integrand in $L_{2}^{n}$ converges to zero as $n$ tends to infinity, for each $u$ and $s$ such that $ a \le u<s \le b$.
On the other hand, we can also bound the 
  rectangular increment $\sigma(y_s)-z^{n}_s-(\sigma(z_{u})-z^{n}_u)$ as follows
\begin{equation}
\left|\sigma(y_s)-z^{n}_s-\left(\sigma(y_u)-z^{n}_u\right)\right|
\le
\left|\sigma(y_s)-\sigma(y_u)\right|+\left|z^{n}_s-z^{n}_u\right|. \label{eq:bnd-rect-incr-2}
\end{equation}
Lemma \ref{lem:psigma}  plus the fact that $y\in\cac_{\eta}^{\ga}$ imply that
\begin{equation*}
\left|\sigma(y_s)-\sigma(y_u)\right|
\lesssim
\left(|y_{s}|^{-\eta}+|y_{u}|^{-\eta}\right)
\left|y_{s}-y_{u}\right|^{\ka+\eta}
\lesssim
\left(|y_{s}|^{-\eta}+|y_{u}|^{-\eta}\right)
(s-u)^{(\ka+\eta)\ga}.
\end{equation*}
Since $(\ka+\eta)\ga>\al$, we get that the term  $\left|\sigma(y_s)-\sigma(y_u)\right|$  is integrable
in $[a,b]^2$ with respect to the measure $\mu(du,ds)=  (s-u) ^{-\al -1} (b-s) ^{\al +\ga -1}\mathbf{1}_{\{u<s\}} duds$ introduced in equation \eqref{eq:def-mu}. Moreover,  the term  
$\left|z^{n}_s-z^{n}_u\right|$ is bounded, up to a constant, by $\Phi^n_{u,s}
+\Psi_{u,s}^n$ (see Lemma \ref{lem:4.7}).
 Applying the dominated convergence theorem as stated in \cite[Theorem 11.4.18]{Roy},
 together with Lemmas \ref{lem:4.8} and~\ref{lem:a4.9},
  we deduce that 
 $L_{2}^{n}$ tends to $0$ as $n $ tends to infinity, which finishes the proof.
\end{proof}

\section{One-dimensional differential equations} 
The purpose of this section is to obtain existence results for the system \eqref{eq:sde-power} in dimension 1, that is for the following equation
\begin{equation}\label{eq:mequ}
y_t=\int_0^t\sigma(y_s)dx_s
,\quad t\ge0.
\end{equation}
We now give a general condition on the coefficient $\si$ in \eqref{eq:mequ}. Notice that a basic example of a function $\sigma$ satisfying Hypothesis  \ref{hyp:nbsigma} below is any power coefficient of the form  $\sigma(\xi)=C|\xi|^{\ka}$, where $\kappa <1$. 

\begin{hypothesis}\label{hyp:nbsigma}
We suppose that $\si :\mathbb{R}\rightarrow  \mathbb{R}$ satisfies Hypothesis \ref{hyp:sigma}, and moreover
\begin{itemize}
\item[(i)]
$\si$ is continuous and   increasing on $\mathbb{R}_{+}$.
\item[(ii)]
$1/\si$ is integrable on compact neighborhoods of zero.
\end{itemize}
\end{hypothesis}

With Hypothesis \ref{hyp:nbsigma} in mind, we shall solve equation \eqref{eq:mequ} thanks to an approximation procedure. We first state the following lemma, whose elementary proof is left to the reader.
\begin{lemma}\label{lem:bar-sigma-m}
For $\si$ satisfying Hypothesis \ref{hyp:nbsigma} and $n\in\mathbb{N}$, define the following function on~$\R$
\begin{equation*}
\si_{n}(\xi) = \left\{ \begin{array}{cc}  \si(\xi), & | \xi | > 2^{-n},   \\  \si(2^{-n}), & |\xi| \leq  2^{-n}.\end{array} \right.
\end{equation*}
Then $\si_{n}$ satisfies (\ref{eq:dsigma}), with 
$ \cn_{\ka,\si_{n}} \le\cn_{\ka,\si}$, where 
$\cn_{\ka,\si}$ is given in (\ref{eq:def-holder-semi-norm}).
\end{lemma}

We shall construct a solution to equation \eqref{eq:mequ} by means of a Lamperti type transformation 
 for $\si$. This transform is classically defined in the following way.

\begin{lemma}\label{prop:invphis}
Let $\si$ be a function fulfilling Hypothesis \ref{hyp:nbsigma} and $\si_{n}$ be defined as in Lemma~\ref{lem:bar-sigma-m}.  For those two functions and $\xi\in\R$, we set
\begin{equation}\label{eq:nphis}
\phi(\xi)=\int_0^{\xi}\frac{ds}{\si(s)}
\quad\mbox{and}\quad
 \phi_{n}(\xi)=\int_0^{\xi}\frac{ds}{ \sigma_{n}(s)}.
\end{equation}
Then $ \phi $ and $ \phi_{n}$ are both   invertible, and  for any  $\xi \in\R$ we have
$| \phi^{-1}(\xi)|   \le   | \phi_{n}^{-1}(\xi)| $, where $\phi^{-1}$, $\phi_{n}^{-1}$ stand for the respective inverse of $\phi$ and $\phi_{n}$.
\end{lemma}

\begin{proof} The result is an immediate consequence of the inequalities 	
$ 	 \phi_{n}  	\le	 \phi $ on $ \mathbb{R}_+$ and 
$ 	 \phi 	\le	{\phi}_{n} $ on $\mathbb{R} _-$,
which follow from our definition (\ref{eq:nphis}).
\end{proof}

 The next result states the uniform (in $n$) Lipschitz regularity of $\phi_{n}^{-1}$.

\begin{lemma}\label{lem:nblip} 
Let $M>0$. Then, there is a constant $ c_M
 > 0$ such that 
$$|\phi_{n}^{-1}(\xi_{1})   -   \phi_{n}^{-1}(\xi_{2})|    \le      c_M
 |\xi_{1} - \xi_{2}|,$$
for all $\xi_1$ and $\xi_2$ such that $|\xi_{1}|, |\xi_{2}| \le  M $ and  for all $ n \in \mathbb{N}$.
\end{lemma}
\begin{proof}   Suppose $|\xi|\le M$. By (\ref{eq:nphis}) and Lemma \ref{prop:invphis}, we get
\begin{equation*}
 \left| \frac{d\phi_{n}^{-1} (\xi)}{d\xi}  \right| =  \left|	\sigma_{n} (\phi_{n}^{-1} (\xi))   \right|	
\le	 \left|	\sigma_{n} (\phi_{n}^{-1} (M)   \right|.
\end{equation*}
In addition, observe that $\lim_{n\to\infty} \sigma_{n} (\phi_{n}^{-1} (M))= \sigma  (\phi ^{-1} (M)$, which means in particular that the sequence $\{\sigma_{n} (\phi_{n}^{-1} (M)), \, n\ge 1\}$ is bounded.
Thus a direct application of the mean value theorem finishes the proof.
\end{proof}

We now proceed to the approximation of equation \eqref{eq:mequ}.    

\begin{proposition}\label{prop:n3}
Suppose that  $x \in \cac^{\ga}([0,T])$ and $x_0=0$.
 Let $ n \in \mathbb{N}$ and $ y_{t}^{n} =  \phi_{n}^{-1}(x_t)$. Then $y^{n}$ solves the following equation
$$ y_{t}^{n} = \int_{0}^{t}	\si_{n}(y_{s}^{n}) \, dx_{s}, 	
\quad\text{for all}\quad 
t\ge 0,
$$
where the integral with respect to $x$ is understood in Young's sense.
\end{proposition}

\begin{proof}
We first  observe that that the function $\phi_{n}^{-1}$  is locally Lipschitz due to Lemma \ref{lem:nblip}.  The function $\si_n$ is also locally Lipschitz according to Lemma  \ref{lem:psigma}.
 Therefore,  
 $\sigma_{n}(\phi_{n}^{-1}(x_s))$ is  locally $\ga$-H\"older continuous. 
 Thus, invoking the usual change of variable in Young's integral (see e.g \cite[Theorem 4.3.1]{Za}) and recalling that $\ga>1/2$, we obtain
$$
y_t^{n}=\int_0^t\sigma_{n}(\phi_{n}^{-1}(x_s))dx_s=\int_0^t\sigma_{n}(y_{s}^{n})dx_s
,\quad t\ge0,
$$
and the proof is complete.
\end{proof}

We now turn to the main result of this section which states the convergence of $y^{n}$ to a solution to equation \eqref{eq:mequ}.

\begin{theorem}   \label{th1v1}
Assume that $\si$ satisfies Hypothesis \ref{hyp:nbsigma}.
Consider       $\eta$   such that $\frac{1-\ga(1+\ka)}{\ga}<\eta<1-\ka$.  
 Let $\phi$ be the function given by (\ref{eq:nphis}), and suppose that
  $x \in \cac^{\ga}([0,T])$ is such that $ |\phi^{-1} (x)|^{-\eta} 	\in	L^{1}([0, T])$
  and $x_0=0$.  
Then the function $y =  \phi^{-1}(x)$	 is a solution of	the equation
$$ y_{t} = \int_{0}^{t} \si( y_{s}) dx_{s}, 
\quad t \ge 0,
$$
where the integral $\int_{0}^{t} \si(y_{s}) dx_{s}$ is understood as in Proposition \ref{prop:integral}.
\end{theorem}

\begin{remark}
Note that $y\equiv0$ is also a solution of equation (\ref{eq:mequ}). So, in general, this equation may have several solutions.
\end{remark}

\begin{proof}[Proof of Theorem \ref{th1v1}]
Let $y^{n}$ be as in Proposition \ref{prop:n3}. 
For each $\xi\in \mathbb{R}$ we have  $\phi_n^{-1}(\xi) \rightarrow  \phi ^{-1}(\xi) $ as $n$ tends to infinity. Hence,   $y^{n}$ converges point-wise to $y$ as $n$ tends to infinity. Therefore, thanks to Proposition \ref{prop:n3}, we are reduced to show that for all $t\ge 0$
$$
I(t):= \lim_{n\rightarrow\infty}\int_0^t \lc \sigma_{n}(y_{s}^{n})- \sigma(y_s) \rc \, dx_s = 0.
$$
Otherwise stated, according to Proposition  \ref{prop:integral}, we have to check that, for $t\ge 0$,
\begin{equation}\label{eq:mira11}
\lim_{n\rightarrow\infty}
\int_0^t
\lc D_{0^+}^\alpha\left(\sigma(y)-\sigma_{n}(y^{n})\right)\rc_{s}
  D_{t^-}^{1-\alpha}x^{t^-}_s  \, ds
= 0,
\end{equation}
where $\alpha$  is such that $1-\ga<\alpha< \ga (\ka +\eta)$.
In order to prove  relation \eqref{eq:mira11}, we first invoke   definition \eqref{eq:def-Da-plus} and relation \eqref{eq:dfracx}. For $s\in[0,T]$, this gives
\begin{equation}\label{eq:dfracdif}
\left|\lc D_{0^+}^\alpha\left(\sigma(y)-\sigma_{n}(y^{n})\right)\rc_{s}
  D_{t^-}^{1-\alpha}x^{t^-}_s \right|
\lesssim  \|x\|_\ga \left(
I_{1,n}(s)+\int_0^sI_{2,n}(s,r)dr\right),
\end{equation}
where
\[
I_{1,n}(s)=
\frac{\left|\sigma(\phi^{-1}(x_s))-\sigma_{n}(\phi^{-1}_{n}(x_s))\right|}
{s^{\alpha}} \
\]
and
\[
I_{2,n}(s,r)=
\frac{\left|\sigma(\phi^{-1}(x_s))-\sigma_{n}(\phi^{-1}_{n}(x_s))
-\left(\sigma(\phi^{-1}(x_r))-\sigma_{n}(\phi^{-1}_{n}(x_r))\right)\right|}
{(s-r)^{1+\alpha}}.
\]
Going back to our aim \eqref{eq:mira11}, we are reduced to prove that
\begin{equation}\label{eq:lim-I1n-I2n}
\lim_{n\to\infty} \iot I_{1,n}(s) \, ds = 0,
\quad\text{and}\quad
\lim_{n\to\infty} \iot \int_{0}^{s} I_{2,n}(s,r) \, dr ds = 0.
\end{equation}
Moreover, thanks to the very definition of $\sigma_n$, we have  that for all $0\le r < s \le t$, 
$I_{1,n}(s)\rightarrow 0$ and $I_{2,n}(s,r)\rightarrow 0$, as    $n\rightarrow \infty$. Our claim \eqref{eq:lim-I1n-I2n} is thus ensured if we can bound $I_{1,n}(s)$ and $I_{2,n}(s,r)$ properly.

Let us start with a bound on the term $I_{1,n}(s)$.  
As in the proof of Lemma  \ref{lem:nblip} we can show that  $I_{1,n}(s)$ is bounded by a constant times $s^{-\alpha}$ for all $n \in \N$. This is enough to apply the dominated convergence theorem.

In order to bound the term $I_{2,n}$, we apply Lemmas \ref{lem:psigma},  \ref{prop:invphis}
 and \ref{lem:nblip}, and the fact that 
 $\sigma_{n}$ satisfies (\ref{eq:dsigma}) with 
 $\cn_{\ka,\si_n} \le \cn_{\ka,\sigma}$ (see Lemma \ref{lem:bar-sigma-m})	 to establish
\begin{eqnarray*}
I_{2,n}(s,r) &\le& (s-r)^{-\alpha -1}\left(| \sigma( \phi^{-1}(x_{s})) - \sigma(\phi^{-1}(x_{r}))|	+ |\sigma_{n}(\phi^{-1}_{n}(x_{s}))  -  \sigma_{n}(\phi_{n}^{-1}(x_{r}))|\right) \\
&\lesssim& (s-r)^{\ga(\ka+\eta) -\alpha -1} 
\left(|\phi^{-1}(x_{s})|^{-\eta}  +  |\phi^{-1}(x_{r})|^{-\eta}	 +  |\phi_{n}^{-1}(x_{s})|^{-\eta}  + |\phi_{n}^{-1}(x_{r})|^{-\eta}\right)				\\		
&\le& (s-r)^{\ga(\ka+\eta) -\alpha -1} 
\left(|\phi^{-1}(x_{s})|^{-\eta}  +  |\phi^{-1}(x_{r})|^{-\eta}
\right).
\end{eqnarray*}
We can thus conclude by the dominated convergence theorem, thanks to the fact that $\ga(\ka+\eta) -\alpha>0$. We get the second claim in \eqref{eq:lim-I1n-I2n}, which completes the proof of our theorem.
\end{proof}

\begin{remark}
A small variant of our calculations also allows to construct a solution to the initial value problem
\begin{equation}\label{eq:sde-with-a}
y_{t} = a + \int_{0}^{t} \si( y_{s}) dx_{s}, 
\quad t \ge 0,
\end{equation}
for a general $a\in\R$. Indeed, along the same lines as for Theorem \ref{th1v1}, one can prove that
 $y_t=\phi^{-1}(x_t+ \phi(a))$ is a solution of \eqref{eq:sde-with-a}
 if $|\phi^{-1}(x_t+ \phi(a))|^{-\eta}\in L^1([0,T])$.
\end{remark}

\section{Multidimensional differential equations}\label{sec:increment-control}

We now turn to the multidimensional setting of equation \eqref{eq:sde-power}. As mentioned in the introduction, our considerations will rely on regularity gain estimates for the solution  when it approaches 0, similarly to \cite{MP,MPS}. Before we deal with these regularity estimates, we will first introduce some new notation.

\subsection{Setting}

In the remainder of the article, we assume that each component $\si^j$, $j=1,\dots, d$  in the coefficients of equation  \eqref{eq:sde-power}, satisfies Hypothesis \ref{hyp:sigma}. As in the previous section, we need an additional hypothesis that says that $\si^j$ behaves as a power function.

\begin{hypothesis}\label{hyp4.1}
We suppose that for each $j=1,\dots,d$,  $\si^j :\mathbb{R}^m\rightarrow  \mathbb{R}^m$ satisfies Hypothesis~\ref{hyp:sigma}, and moreover:
\begin{itemize}
\item[(i)] For  any $\xi \in \R^m$ we have $|\si^j (\xi)| \gtrsim    |\xi|^\kappa$.
\item[(ii)]  $\sigma^j$ is differentiable with $\nabla \si^j$ locally H\"older continuous of order larger than $\frac 1\gamma-1$ in the set
$\{|\xi| \not =0\}$.
\end{itemize}
\end{hypothesis}

Fix $a\in \R^m$, $a\not =0$, and 
we consider equation 
\begin{equation}\label{eq:sde-power-approx}
y_{t} = a  + \sum_{j=1}^{d}\int_{0}^{t} \si^{j}(y_{u}) \, dx_{u}^{j}, \quad t \in [0,T].
\end{equation}
Using an approximation of $\si^j$ similar to  Lemma \ref{lem:bar-sigma-m} and applying known results on existence and uniqueness of solutions to equations driven by H\"older continuous functions  (see e.g \cite{FV-bk}), it is easy to show the following result.

\begin{proposition}  \label{th1} Suppose that Hypothesis \ref{hyp4.1} (ii) holds, and let $T$ be a given strictly positive time horizon. Then, there exists a continuous function $y$ defined on $[0,T]$ and an instant $\tau \le T$, such that one of the following two possibilities holds:  

\begin{itemize}
\item[(A)] $\tau =T$, $y$ is nonzero on $[0,T]$,  $y\in\cac^{\ga}([0,T];\R^m)$   and $y$ solves equation  \eqref{eq:sde-power-approx} on $[0,T]$, where the integrals $\int \si^{j}(y_{u}) \, dx_{u}^{j}$ are understood in the usual Young sense.

\item[(B)]   We have $\tau<T$. Then for any $t<\tau$, the path $y$ sits in $\cac^{\ga}([0,t];\R^m)$   and $y$ solves equation  \eqref{eq:sde-power-approx} on $[0,t]$.
Furthermore,  $y_s \not=0$ on   $[0,\tau)$, $\lim_{t\rightarrow \tau} y_t=0$ and $y_t=0$ on the interval $[\tau,T]$.
\end{itemize}
\end{proposition}

Notice that our option (A) above leads to classical solutions of equation \eqref{eq:sde-power-approx}. 
In the rest of this section, we will assume (B), that is the function $y$ given by Proposition \ref{th1}  vanishes  in the interval $[\tau, T]$. Our aim is thus to prove the following two facts:
\begin{itemize}
\item 
The path $y$ is globally $\ga$-H\"older continuous on $[0,T]$.

\item
The integrals $\int \si^{j}(y_{u}) \, dx_{u}^{j}$ can be understood as limits of Riemann sums, and $y$ solves equation \eqref{eq:sde-power-approx} on $[0,T]$.
\end{itemize}
Notice that in order to achieve this aim, we will need some additional hypotheses  on $x$. We shall also assume $\ga+\ka>1$, which is a natural condition in our context (as explained in the introduction).

In order to quantify the regularity gain of solutions close to the origin, we split the interval $[0,\tau)$ as follows. We first define $a_{q}=2^{-q}$ and we introduce a decomposition of the  space $\R_{+}$, which is the state space for $|y|$, into the following sets:
\[
I_{-1}=\left[1,\infty\right), \quad {\rm and} \quad  I_{q}=[a_{q+1}, a_{q}),  \quad q\ge 0.
\]
We also need  to define the intervals:
\[
J_{-1}= \left[3/4, \infty\right), \quad  {\rm and} \quad
J_{q} = \lc \frac{a_{q+2}+a_{q+1}}{2}, \frac{a_{q+1}+a_{q}}{2}  \rp
=: \lc \hat{a}_{q+1}, \hat{a}_{q} \rp, \quad q\ge 0.
\]
Notice that $\hat{a}_{q}=\frac{3}{2^{q+2}}$. We now construct a partition of $[0,\tau)$ as follows. Assume that $|a|\in I_{q_{0}}$, and set $\la_{0}=0$ and 
\begin{equation*}
\tau_{0} = \inf\lcl t\ge 0 :\, |y_{t}|\not\in I_{q_{0}} \rcl.
\end{equation*}
In this case $y_{\tau_{0}}\in J_{{\hat q}_{0}}$ with ${\hat q}_0\in\{q_0,q_0-1\}$.
We then set:
\begin{equation*}
\la_{1} = \inf\lcl t\ge \tau_{0} :\, |y_{t}|\not\in J_{{\hat q}_{0}}\rcl.
\end{equation*}
In this way we recursively construct a sequence of stopping times $\la_{0}<\tau_{0}<\cdots<\la_{k}<\tau_{k}$ such that
\begin{equation}\label{eq:def-sigma-tau-k}
|y_{t}| \in \lc \frac{b_1}{2^{q_{k}}}, \, \frac{b_2}{2^{q_{k}}} \rc,
\quad\text{for}\quad
t\in [\la_{k},\tau_{k}] \cup [\tau_{k},\la_{k+1}],
\end{equation}
where $b_1=\frac 38$, $b_2=\frac 34$
and $q_{k+1}=q_{k}+\ell$, with $\ell\in\{-1,0,1\}$, assuming that $q_k\ge 1$. Notice that if $q_k =0$ or $q_k=1$, then the upper bound $b_2$ may be infinity. This construction is depicted in Figure \ref{fig:stop-times}.

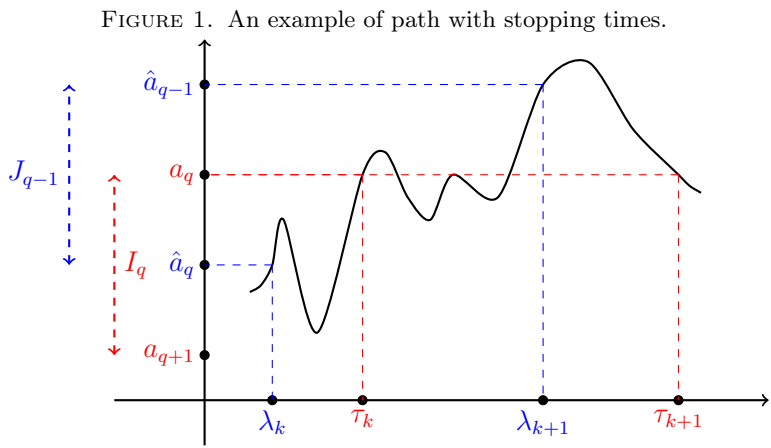
\begin{figure}[htbp]
\begin{center}
\caption{An example of path with stopping times.}
\label{fig:stop-times}

\begin{tikzpicture}[xscale=0.3,yscale=0.3]

\draw [->,  thick] (-1,4) --(28,4);
\draw [->,  thick] (3,2) --(3,20);

\draw [fill,black] (6,4) circle [radius=.2];
\node [below,scale=.9, blue] at (6,4) {$\la_{k}$}; 

\draw [fill,black] (10,4) circle [radius=.2];
\node [below,scale=.9, red] at (10,4) {$\tau_{k}$}; 

\draw [fill,black] (18,4) circle [radius=.2];
\node [below,scale=.9, blue] at (18,4) {$\la_{k+1}$};

\draw [fill,black] (24,4) circle [radius=.2];
\node [below,scale=.9, red] at (24,4) {$\tau_{k+1}$}; 


\draw [fill,black] (3,6) circle [radius=.2];
\node [left,scale=.9, red] at (3,6) {$a_{q+1}$}; 

\draw [fill,black] (3,10) circle [radius=.2];
\node [left,scale=.9, blue] at (3,10) {$\hat{a}_{q}$}; 

\draw [fill,black] (3,14) circle [radius=.2];
\node [left,scale=.9, red] at (3,14) {$a_{q}$}; 

\draw [fill,black] (3,18) circle [radius=.2];
\node [left,scale=.9, blue] at (3,18) {$\hat{a}_{q-1}$}; 

\draw [<->,  thick,red, dashed] (-1,6) --(-1,14);
\node [right,scale=.9, red] at (-1,10) {$I_{q}$}; 

\draw [<->,  thick,blue, dashed] (-3,10) --(-3,18);
\node [left,scale=.9, blue] at (-3,14) {$J_{q-1}$}; 

\draw [black, thick] plot [smooth] 
coordinates 
{(5,8.8) (5.5, 9.1) 
(6,10) (6.5,12) (8,7) (10,14)
 (11,15) (12,13) (13,12) (14,14) (16, 13) (18,18)
 (20,19) (22,16) (24,14)
 (24.5, 13.5) (25, 13.2)};

\draw [dashed,blue] (6,10) --(6,4);
\draw [dashed,blue] (6,10) --(3,10);

\draw [dashed,red] (10,14) --(10,4);
\draw [dashed,red] (10,14) --(3,14);

\draw [dashed,blue] (18,18) --(18,4);
\draw [dashed,blue] (18,18) --(3,18);

\draw [dashed,red] (24,14) --(24,4);
\draw [dashed,red] (24,14) --(3,14);

\end{tikzpicture}

\end{center}
\end{figure}

Finally, let us justify a simplification in notations which will prevail until the end of this Section.
\begin{remark}
Notice that, owing to our Hypothesis \ref{hyp:sigma}, our problem relies heavily on radial variables in $\R^{m}$. Therefore, in order to alleviate vectorial notations, we will carry out the computations below for $m=d=1$. This allows us in particular to drop the exponents $j$ in our formulae. The reader will easily generalize our considerations to higher dimensions.
\end{remark}

\subsection{Regularity estimates}\label{sec:reg-estimates}
Let us start with a decomposition lemma for the solution to the regularized equation~\eqref{eq:sde-power-approx}. We recall a convention which will prevail until the end of the paper: for a function $f$ defined on $[0,T]$, we set $\delta f_{st}=f_{t}-f_{s}$.

\begin{lemma}
Let $0\le s < t < \tau$.
For $l\ge 0$ we consider the dyadic partition $\Pi_{st}^{l}$ of $[s,t]$ defined by $t_{i}^{l}=s+2^{-l}i(t-s)$ for $l\ge 0$ and $i=0,\ldots,2^{l}$. 
Then one can write:
\begin{equation}\label{eq:dyadic-dcp-integral}
\delta y_{st}
=
\si(y_{s}) \, \der x_{st} + \sum_{l=1}^{\infty} K_{st}^{l},
\end{equation}
where
\begin{equation*}
K_{st}^{l} = \sum_{i=0}^{2^{l}-1} 
\lc \der\si(y) \rc_{t_{2i}^{l+1}t_{2i+1}^{l+1}}
\der x _{t_{2i+1}^{l+1}t_{2i+2}^{l+1}}.
\end{equation*}
\end{lemma}

\begin{proof}
Since $s,t\in[0,\tau)$, the integral $\int_{s}^{t} \si(y_{u}) \, dx_{u} $ is a usual Young integral, which is thus limit of Riemann sums along dyadic partitions. Let us write $J_{st}^{l}$ for those Riemann sums, and notice that
\begin{eqnarray}
J_{st}^{l} &=& 
\sum_{i=0}^{2^{l}-1} \si\big( y_{t_{i}^{l}}\big) \, \der x_{t_{i}^{l}t_{i+1}^{l}}  \label{eq:exp1-Jl}\\
&=&
\sum_{i=0}^{2^{l}-1} 
\si\big( y_{t_{2i}^{l+1}}\big) 
\lc \der x_{t_{2i}^{l+1} t_{2i+1}^{l+1}} + \der x_{t_{2i+1}^{l+1} t_{2i+2}^{l+1}}\rc. \label{eq:exp2-Jl}
\end{eqnarray}
Then, we know from usual Young integration that $J_{st}^{l}$ converges, as $l\to\infty$, to $\int_{s}^{t} \si(y_{u}) \, dx_{u}$. Therefore, we can write
\begin{equation*}
\int_{s}^{t} \si(y_{u}) \, dx_{u}
=
\si(y_{s}) \, \der x_{st}  + \sum_{l=0}^{\infty} \lp J_{st}^{l+1} - J_{st}^{l} \rp .
\end{equation*}
Resorting to expression \eqref{eq:exp1-Jl} for $J_{st}^{l+1}$ and to expression \eqref{eq:exp2-Jl} for $J_{st}^{l}$ above, some elementary algebraic manipulations reveal that $J_{st}^{l+1} - J_{st}^{l}=K_{st}^{l}$, which ends the proof.
\end{proof}

Let us state an additional (harmless) hypothesis on our noise $x$, which will be crucial in order to get sharp regularity estimates.
\begin{hypothesis}\label{hyp:reg-x-gamma-gamma1}
There exists $\ep_{1}>0$ such that for $\ga_{1}=\ga+\ep_{1}$, we have  $\|x\|_{\ga_{1}}<\infty$ and $\ga_1+\ga\ka<1$. 
\end{hypothesis}

We are now ready to give the basis of the strategy alluded to above, based on a regularity gain when $y$ is close to 0.

\begin{proposition}\label{prop:regularity-gain}
Assume $\si$ satisfies Hypothesis \ref{hyp4.1} and $x$ is such that \ref{hyp:reg-x-gamma-gamma1} is fulfilled.  Then the following bounds hold true:

\noindent
\emph{(i)}
There exist  constants  $c_{0,x}$  and $c_{1,x}$ such that for $s,t\in [\la_{k}, \la_{k+1})$ satisfying
\begin{equation}\label{eq:cdt-small-scale}
|t-s|\le  c_{0,x} \, 2^{-\al q_{k}},
\quad\text{with}\quad
\al := \frac{1-\ka}{\ga},
\end{equation}
we have the following  bound:
\begin{equation}\label{eq:a-priori-bnd-yn}
\lln  \der y_{st} \rrn \le
 c_{1,x} \, 2^{-q_{k}\ka} |t-s|^{\ga}   .
\end{equation}

\noindent
\emph{(ii)}
With Hypothesis \ref{hyp:reg-x-gamma-gamma1} in mind,  we get a refined decomposition for $\der y_{st}$. Namely, if $s,t$ are two instants in $[\la_{k}, \la_{k+1})$ such that \eqref{eq:cdt-small-scale} holds true, we have the following relation for $\der y_{st}$:
\begin{equation}\label{eq:refined-dcp-yn}
\der y_{st}
=
\si(y_{s}) \, \der x_{st} + r_{st},
\quad\text{with}\quad
|r_{st}| \le   c_{2,x} \, 2^{-\ka_{\ep_{1}} q_{k}} |t-s|^{\ga},
\end{equation}
where we have set $\ka_{\ep_{1}}=\ka+\ep_{1}\al$.
\end{proposition}

\begin{proof}
For $k\ge 1$ and $\nu>0$ we set
\begin{equation*}
\|y\|_{\ga,k,\nu}
=\sup\lcl
\frac{|\der y_{uv}|}{|v-u|^{\ga}} : \,
u,v\in [\la_{k}, \la_{k+1}), \, |v-u|\le \frac{c_{0}}{2^{\nu}}
\rcl,
\end{equation*}
where the constants $c_{0}$ and $\nu$ will be tuned on later.

\noindent
\emph{Step 1: Proof of \eqref{eq:a-priori-bnd-yn}.}
Pick $s,t\in [\la_{k}, \la_{k+1})$ such that $|s-t|\le c_{0} 2^{-\nu}$.
Recall that we consider the dyadic partitions of $[s,t]$, with $t_{i}^{l}=s+2^{-l}i(t-s)$ for $l\ge1$ and $i=0,\ldots,2^{l}$.
Start from decomposition~\eqref{eq:dyadic-dcp-integral}.
Then, since  both $|y_{s}|$ and $|y_{t}|$ lie into $[b_1 2^{-q_k}, b_2 2^{-q_k}]$ and $\si$ verifies Hypothesis 2.1, we obviously have
\begin{equation}\label{eq:bnd-main-increment}
\lln \si(y_{s}) \, \der x_{st} \rrn
\le
 c_1\| x\|_\gamma |t-s|^{\ga}2^{-q_k\ka},
\end{equation}
where $c_1 = \cn_{\ka,\si}  b_2^\ka   $.

In the remainder of this proof, we denote $t^{l+1}_{2i},t^{l+1}_{2i+1}$ by
$t_{2i},t_{2i+1}$, respectively, to simplify the notation.
We now bound the quantity $[\der\si(y)]_{t_{2i}t_{2i+1}}  \der x_{t_{2i+1}t_{2i+2}}$ popping up in~\eqref{eq:dyadic-dcp-integral}. Thanks to Lem\-ma~\ref{lem:psigma}, for any $\eta\le 1-\ka$ we have
\begin{equation*}
\lln [\der\si(y)]_{t_{2i}t_{2i+1}} \rrn
\le \cn_{\ka,\si}  
\lp |y_{t_{2i}}|^{-\eta} +|y_{t_{2i+1}}|^{-\eta} \rp \lln  \der y_{t_{2i}t_{2i+1}} \rrn^{\ka+\eta}.
\end{equation*}
Thus, since $y_{t_{2i}},y_{t_{2i+1}}\in [b_1 2^{-q_k}, b_2 2^{-q_k}]$ we get
\begin{equation}\label{eq:bnd-remainder-1}
\lln [ \der\si(y)]_{t_{2i}t_{2i+1}} \rrn \lln \der x_{t_{2i+1}t_{2i+2}}  \rrn
\le \cn_{\ka,\si} 2b_1^{-\eta}
\|x\|_{\ga} \|y\|_{\ga,k,\nu}^{\ka+\eta}  \, 2^{q_k\eta} \lln  \frac{t-s}{2^{l}} \rrn^{(1+\ka+\eta)\ga}.
\end{equation}
We choose $\eta$ above such that $\ga(1+\ka+\eta)=2\ga$. It is readily checked that such a $\eta$ verifies
\begin{equation*}
\eta= 1-\ka.
\end{equation*}
Furthermore, with this value of $\eta$ in hand, relation \eqref{eq:bnd-remainder-1} becomes
\begin{equation}\label{eq:pdct-increm-sigma-y-x}
\lln [\der\si(y)]_{t_{2i}t_{2i+1}} \rrn \lln \der x_{t_{2i+1}t_{2i+2}}  \rrn
\le \cn_{\ka,\si} 2b_1^{\ka-1}
\|x\|_{\ga} \|y\|_{\ga,k,\nu}  \, 2^{q_k(1-\ka)} \lln  \frac{t-s}{2^{l}} \rrn^{2\ga}.
\end{equation}
Plugging this inequality into the terms $K_{st}^{l}$ of \eqref{eq:dyadic-dcp-integral} we end up with
\begin{equation}\label{eq:bnd-remainder-2}
\sum_{l=1}^{\infty} | K_{st}^{l} |
\le
c_{3,x} \|y\|_{\ga,k,\nu} \, 2^{q_k (1-\ka)} \, |t-s|^{2\ga},
\end{equation}
where we have set $c_{3,x}=  \frac {  \cn_{\ka,\si} 2b_1^{\ka-1}}{2^{2 \ga }-1} \|x\|_{\ga}$.
Reporting \eqref{eq:bnd-main-increment} and \eqref{eq:bnd-remainder-2} into \eqref{eq:dyadic-dcp-integral}, this yields
\begin{equation}\label{eq:estim-increments-y-small}
|\der y_{st}|
\le
 c_1\| x\|_\gamma   |t-s|^{\ga}2^{-q_k\ka} + A_{st}^{2},
 \quad\text{with}\quad
 A_{st}^{2}=
 c_{3,x}    \|y\|_{\ga,k,\nu} \, 2^{q_k (1-\ka)} \, |t-s|^{2\ga}.
\end{equation}
We should now bound the term $A_{st}^{2}$ as a $\ga$-H\"older increment. Indeed, recalling that we assume $|t-s|\le  c_{0} 2^{-\nu}$, we get
\begin{equation}\label{eq:upp-bnd-A2}
A_{st}^{2} \le c_{3,x} c_{0}^{\ga} \, 2^{q_k (1-\ka)- \nu\ga}   \|y\|_{\ga,k,\nu} |t-s|^{\ga}.
\end{equation}
We now choose $c_{0}$ and $\nu$ so that $c_{3,x} c_{0}^{\ga} \, 2^{q_k (1-\ka)- \nu\ga}\le \frac12$. It is readily checked that this is achieved for $c_{0}$ small enough and $\nu=\al q_k:=\ga^{-1}(1-\ka)q_k$ given by \eqref{eq:cdt-small-scale}. With those values of $c_{0}$ and $\nu$ in hand, relation \eqref{eq:estim-increments-y-small} becomes
\begin{equation*}
\|y\|_{\ga,k,\nu} \le 
c_1\| x\|_\gamma  2^{-q_k\ka}  +  \frac12 \|y\|_{\ga,k,\nu} ,
\end{equation*}
from which \eqref{eq:a-priori-bnd-yn} is easily deduced, with $c_{1,x}= 2c_1 \|x\|_\ga$.

\noindent
\emph{Step 2: Proof of \eqref{eq:refined-dcp-yn}.}
Go back to relation \eqref{eq:pdct-increm-sigma-y-x} and invoke Hypothesis \ref{hyp:reg-x-gamma-gamma1} in order to get
\begin{equation*}
\lln [\der\si(y)]_{t_{2i}t_{2i+1}} \rrn \lln \der x_{t_{2i+1}t_{2i+2}}  \rrn
\le \cn_{\ka,\si} 2b_1^{\ka-1}
\|x\|_{\ga_{1}} \|y\|_{\ga,k,\nu}  \, 2^{q_k(1-\ka)} \lln  \frac{t-s}{2^{l}} \rrn^{2\ga+\ep_{1}}.
\end{equation*}
Moreover, according to~\eqref{eq:dyadic-dcp-integral}, the term $r_{st}$ in \eqref{eq:refined-dcp-yn} is given by $\sum_{l=1}^{\infty}  K_{st}^{l}$. Proceeding as for relations \eqref{eq:bnd-remainder-2} and \eqref{eq:estim-increments-y-small}, we obtain that 
\begin{equation}\label{eq:bnd-remainder-3}
\lln r_{st}\rrn \le \sum_{l=1}^{\infty} \lln K_{st}^{l} \rrn \le A_{st}^{2}=
\tilde{ c}_{3,x}    \|y\|_{\ga,k,\nu} \, 2^{q _k(1-\ka)} \, |t-s|^{2\ga+\ep_{1}},
\end{equation}
where
$\tilde{ c}_{3,x} = \frac {  \cn_{\ka,\si} 2b_1^{\ka-1}}{2^{2 \ga+\ep_1 }-1} \|x\|_{\ga_1}$.

We now plug the a priori bound~\eqref{eq:a-priori-bnd-yn} on $\|y\|_{\ga,k,\nu}$ we have just obtained, and read the regularity of $A^{2}$ in $\ga$-H\"older norm. Similarly to \eqref{eq:upp-bnd-A2}, we can recast \eqref{eq:bnd-remainder-3} as:
\[
A_{st}^{2} \le  \tilde{c}_{3,x}   c_0^{\ga +\ep_1} 
 2^{q_k (1-\ka)-\nu(\ga+\ep_{1})} \, c_{1,x} \, 2^{-q_{k}\kappa}  |t-s|^{\ga}.
 \]
 Let us recall that $\nu =\alpha q_k$. Therefore we obtain:
\begin{equation*}
A_{st}^{2} \le  
\tilde{ c}_{3,x}c_0^{\ga+\varepsilon_1}c_{1,x}
 2^{-q_k (\ka+\al \ep_{1})\ga} |t-s|^{\ga}.
\end{equation*}
Taking into account the fact that $\ka_{\ep_{1}}=\ka+\al \ep_{1}$, this finishes the proof of \eqref{eq:refined-dcp-yn}.
\end{proof}

In the sequel we shall need some regularity estimates for $y$ on time scales slightly larger than $2^{-\al q_{k}}$ with $\al= \ga^{-1}(1-\ka)$. This is the contents of the following property.

\begin{corollary}\label{cor:a-priori-bnd-yn-larger-intv}
Under the same hypotheses as in Proposition \ref{prop:regularity-gain},   consider $\ep_{2}>0$ such that 
\[
\ep_2 < \min \left(  \gamma^{-1}(1-\ka), \ka (1-\ga)^{-1}, \frac {\ka + \ga^{-1}(1-\ka)\ep_1}{1+\ep_1} \right).
\]
 Then there exists a  constant  $c_{4,x}=2^{1-\ga}c_{0,x}$  such that for $s,t\in [\la_{k}, \la_{k+1})$ satisfying $|t-s|\le  c_{4,x} 2^{-(\al-\ep_{2}) q_{k}} $ with $\al= \ga^{-1}(1-\ka)$ we have
\begin{equation}\label{eq:a-priori-bnd-yn-larger-intv}
\lln  \der y_{st} \rrn \le
 c_{5,x} 2^{-q_{k}\ka_{\ep_{2}}^{-}} 
 |t-s|^{\ga}   ,
\quad\text{with}\quad
\ka_{\ep_{2}}^{-} = \ka- (1-\ga)\ep_{2}.
\end{equation}
 Moreover, under the same conditions on $s,t$, decomposition \eqref{eq:refined-dcp-yn} still holds true, with
\begin{equation}\label{eq:refined-dcp-yn-larger-intv}
|r_{st}| \le c_{6,x}   2^{-q_{k} \ka_{\ep_{1},\ep_{2}}}  |t-s|^{\ga},
\quad\text{where}\quad
\ka_{\ep_{1},\ep_{2}} = \ka+\al\ep_{1}-\ep_{2} - \ep_{1}\ep_{2}.
\end{equation}

\end{corollary}

\begin{proof}
We take up the notation introduced for the proof of Proposition \ref{prop:regularity-gain}, and we split again our computations in 2 steps.

\noindent
\emph{Step 1: Proof of \eqref{eq:a-priori-bnd-yn-larger-intv}.}
Start from inequality \eqref{eq:a-priori-bnd-yn}, which is valid for $|t-s|\le  c_{0,x} 2^{-\al q_{k}}$. Now let $m\in\N$ and consider $s,t\in [\la_{k}, \la_{k+1})$ such that $c_{0,x}  (m-1) 2^{-\al q_{k}} <|t-s|\le  c_{0,x}  m 2^{-\al q_{k}}$. We partition the interval $[s,t]$ by setting $t_{j}=s+c_{0,x} j 2^{-\al q_{k}}$ for $j=0,\ldots,m-1$ and $t_{m}=t$. Then we simply write
\begin{equation*}
|\der y_{st}|
\le
\sum_{j=0}^{m-1} |\der y_{t_{j}t_{j+1}}|
\le
c_{1,x}  2^{-q_{k}\ka} \sum_{j=0}^{m-1} \lp t_{j+1} - t_{j} \rp^{\ga}
\le
c_{1,x} 2^{-q_{k}\ka} m^{1-\ga} |t-s|^{\ga},
\end{equation*}
where the last inequality stems from  the fact that $t_{j+1} - t_{j}\l \le (t-s)/m$. Now  the upper bound \eqref{eq:a-priori-bnd-yn-larger-intv} is easily deduced by applying the above inequality to $m=[2^{\ep_{2} q_k}]+1$.

\noindent
\emph{Step 2: Proof of \eqref{eq:refined-dcp-yn-larger-intv}.} Once \eqref{eq:a-priori-bnd-yn-larger-intv} is proven, we go again through the estimation of $K_{st}^{l}$. Replacing $\|y\|_{\ga,k,\nu}$ by $c_{5,x} 2^{-q_{k}\ka_{\ep_{2}}^{-}}$ in \eqref{eq:bnd-remainder-3}, we end up with
\begin{equation*}
|r_{st}| 
\le
c_{6,x} \, 2^{-q_k \ka_{\ep_{2}}^{-}} 2^{q_k (1-\ka)} 2^{-q_k (\al
-\varepsilon_2)(\ga+\ep_{1})} |t-s|^{\ga}
= c_{6,x} \,  2^{-q_k \ka_{\ep_{1},\ep_{2}}}  |t-s|^{\ga},
\end{equation*}
which is our claim \eqref{eq:refined-dcp-yn-larger-intv}.

\end{proof}

\subsection{Estimates for stopping times}

Thanks to the regularity estimates of the previous section, we get a bound on the difference $\la_{k+1}-\la_{k}$ which roughly states that a solution to  equation  \eqref{eq:sde-power-approx},   cannot go too sharply to 0.
\begin{proposition}\label{prop:upper-bound-diff-sigma-k}
The sequence of stopping times $\{\la_{k},\, k\ge 1\}$ defined by \eqref{eq:def-sigma-tau-k} satisfies
\begin{equation}\label{eq:low-bnd-increment-sigma-k}
\la_{k+1}-\la_{k} \ge  c_{7,x}  \, 2^{-\al q_{k}},
\end{equation}
where we recall that $\al=(1-\ka)/\ga$.
\end{proposition}

\begin{proof}
 We shall prove that $\tau_{k}-\la_{k}$ satisfies a lower bound of the form
\begin{equation}\label{eq:low-bnd-increment-tau-k-sigma-k}
\tau_{k}-\la_{k} \ge  c_{7,x} \, 2^{-\al q_k}.
\end{equation}
Along the same lines we can prove a similar bound for $\la_{k+1}-\tau_{k}$, and this will prove our claim~\eqref{eq:low-bnd-increment-sigma-k}. 

Inequality \eqref{eq:low-bnd-increment-tau-k-sigma-k} is obtained in the following way. We observe that in order to get out of the interval $[\la_{k},\tau_{k})$, an increment of size $2^{-(q_k+1)}$ must occur.  Indeed, at $\lambda_k$ the solution is at the middle point of  $I_{q_k}$ and the length of this interval is of order $2^{-q_k}$.
However, relation~\eqref{eq:a-priori-bnd-yn} asserts that if $|\der y_{st}|\ge 2^{-(q_{k}+1)}$ and $|t-s|\le c_{0,x}  2^{-\al q_k} $,
 then we must have
\begin{equation}\label{eq:lower-bnd-t-s}
c_{1,x}  \frac{|t-s|^{\ga}}{2^{\ka q_k}} \ge \frac{1}{2^{q_k+1}},
\end{equation}
which implies  
\[
 |t-s | \ge  \left(2 c_{1,x} \right)^{-\frac 1\ga} 2^{-\frac{(1-\ka)q_k}{\ga}}
 = \left(2c_{1,x} \right)^{-\frac 1\ga} 2^{-\al q_k}
 .
 \]
This finishes our proof.
\end{proof}

In order to sharpen Proposition \ref{prop:upper-bound-diff-sigma-k}, we introduce a roughness hypothesis on $x$, borrowed from \cite{CHLT}. As we shall see, this assumption is satisfied when $x$ is a fractional Brownian motion.

\begin{hypothesis}\label{hyp:roughness}
We assume that for $\hep$ arbitrarily small there exists  
a constant $c>0$ such that for every $s$ in $\left[  0,T\right]  $, every
$\epsilon$ in $(0,T/2]$, and every $\phi$ in $\mathbb{R}^{d}$ with $\left\vert
\phi\right\vert =1$, there exists $t$ in $\left[  0,T\right]  $ such that
$\epsilon/2<\left\vert t-s\right\vert <\epsilon$ and
\[
\left\vert \left\langle \phi,\der x_{st}\right\rangle \right\vert >c \, \epsilon^{\ga+\hep}.
\]
The largest such constant is called the modulus of $(\ga+\hep)$-H\"{o}lder
roughness of $x$, and is denoted by $L_{\ga,\hep}\left(  x\right)$.
\end{hypothesis}

Under this hypothesis, we are also able to upper bound the difference $\la_{k+1}-\la_{k}$ in a useful way

\begin{proposition}\label{prop:bound-diff-sigma-k-2}
For all $\ep_{2}<\frac{\al\varepsilon_1}{1+\ga+\varepsilon_1}\wedge
\frac{\kappa}{1-\ga}$
 and $q_{k}$ large enough (that is for $k$ large enough, since $\lim_{k\to\infty}q_k= \infty$ under Assumption (B) of Proposition \ref{th1}), the sequence of stopping times $\{\la_{k},\, k\ge 1\}$ defined by \eqref{eq:def-sigma-tau-k} satisfies
\begin{equation}\label{eq:upp-bnd-increment-sigma-k}
\la_{k+1}-\la_{k}\le  c_{x,\ep_{2}}  2^{- q_{k} (\al-\ep_{2})},
\end{equation}
where we recall that $\al=(1-\ka)/\ga$. Furthermore, inequality \eqref{eq:a-priori-bnd-yn-larger-intv} can be extended as follows:  there exists a constant $c_{x}$ such that for $s,t\in [\la_{k}, \la_{k+1})$ we have
\begin{equation}\label{eq:a-priori-bnd-yn-whole-interval-lambda}
\lln  \der y_{st} \rrn \le
 c_{x} 2^{-\ka_{\ep_{2}}^{-}q_{k}}  |t-s|^{\ga}   .
\end{equation}
\end{proposition}

\begin{proof}
If  \eqref{eq:upp-bnd-increment-sigma-k} does not hold,  this implies that there exists $\ep_2<\frac{\al\varepsilon_1}{1+\ga+\varepsilon_1}\wedge
\frac{\kappa}{1-\ga}$ 
satisfying the condition of Corollary   \ref{cor:a-priori-bnd-yn-larger-intv} so that for any constant $C$ the inequality
\begin{equation}\label{a1}
\la_{k+1}-\la_{k}\ge  C  2^{- q_{k} (\al-\ep_{2})}
\end{equation}
holds for infinitely many values of $k$.
This implies that
\begin{equation}\label{a2}
\la_{k+1}-\la_{k} \ge  C \, 2^{-q_k(1-\ka)/(\ga+\hep)},
\end{equation}
if we  choose $\hep$ small enough so that  $(1-\ka)/(\ga+\hep)\ge \al-\ep_{2}$. We wish to exhibit a contradiction, namely that one can find $s,t\in[\la_{k},\la_{k+1}]$ such that $|\delta y_{st}| > |J_{q_k}|$, where $|J_{q_k}|$ denotes the size of $J_{q_k}$.

In order to lower bound $|\delta y_{st}|$, let us first invoke Hypothesis \ref{hyp:roughness}. 
Since our computations are performed in the one-dimensional case for notational sake, we can in fact recast Hypothesis~\ref{hyp:roughness} as follows. Choose
\[
\ep:=  \frac{c_{1} \, 2^{-\frac{q_k(1-\ka)}{\ga+\hep}}}{\lc L_{\ga,\hep}(x) \rc^{\frac{1}{\ga+\hep}}}  \le  C \, 2^{-\frac{q_k(1-\ka)}{\ga+\hep}},
\]
which can be achieved by taking the constant $C$ large enough, for a given constant $c_1$. Then there exist $s,t\in [\la_{k},\la_{k+1}]$ satisfying
\begin{equation}\label{eq:low-bnd-inc-x}
\frac \ep 2 \le |t-s| 
\le  \ep,
\quad\text{and}\quad
|\der x_{st}| \ge  \, c_{1}^{\ga+\hep} \, 2^{-q_k(1-\ka)}.
\end{equation}
Notice that $c_{1}$ can be made arbitrarily large, by playing with $k$ and $\hep$. In addition, we can use the fact that $|\si(y_{s})|\ge  c2^{-q_k\ka}$ whenever $s\in[\la_{k},\la_{k+1}]$ . Indeed, this follows from Hypothesis~\ref{hyp4.1} and the fact that $y_s \ge b_1 2^{-q_k} \ge 2^{ -q_k-2}$.
 This entails, for $s,t$ as in \eqref{eq:low-bnd-inc-x}
\begin{equation*}
|\si(y_{s}) \der x_{st}| \ge   c     c_{1}^{\ga+\hep} \, 2^{-q_k}.
\end{equation*}
If \eqref{a1} holds true, we can now choose $c_{1}$  so that $c c_{1}^{\ga+\hep}\ge 6$. This yields
\begin{equation*}
|\si(y_{s}) \der x_{st}| \ge   6\cdot  2^{-q_k} = 2 |J_{q_k}|.
\end{equation*}
In particular the size of this increment is larger than twice the size of $J_{q_k}$.

We now assume again that we have chosen $\hep$ small enough so that  $(1-\ka)/(\ga+\hep)\ge \al-\ep_{2}$. Then the upper bound on $|t-s|$ in \eqref{eq:low-bnd-inc-x}  also implies $|t-s|\le c_{8,x}  2^{-q_k (\al-\ep_{2})}$.
For the two instants $s,t$ exhibited in relation \eqref{eq:low-bnd-inc-x}, we resort to decomposition~\eqref{eq:dyadic-dcp-integral} together with the bound~\eqref{eq:refined-dcp-yn-larger-intv}. This yields
\begin{equation*}
|\der y_{st}| \gtrsim A_{st}^{1} - A_{st}^{2},
\quad\text{with}\quad
A_{st}^{1} = 6 \cdot 2^{-q_k},
\quad
A_{st}^{2} \le c_{6,x} 2^{-q_k \ka_{\ep_{1},\ep_{2}}}   |t-s|^{\ga} \le c_{9,x} 2^{-q_k \mu_{\ep_{2}}},
\end{equation*}
where we recall that $\ka_{\ep_{1},\ep_{2}}=\ka+\al\ep_{1}-\ep_{2}- \ep_{1}\ep_{2}$ and where we obtain 
\begin{equation*}
\mu_{\ep_{2}}= \ka_{\ep_{1},\ep_{2}} +(\al-\ep_{2})\ga
= 1+\al\ep_{1} -(1+\ga+\ep_{1})\ep_{2}.
\end{equation*}
Our aim is now to prove that $A_{st}^{2}$ can be made negligible with respect to $2^{-q_k}$ when $q_k$ is large enough. This is achieved whenever $\mu_{\ep_{2}}>1$, and this condition can be met by picking $\ep_{1}$ large enough and $\ep_{2}$ small enough.
Summarizing our considerations, we have thus shown that $A_{st}^{1}$  is larger than twice $|J_{q_k}| =3\cdot 2^{-q_k}$ and that $A_{st}^{2}$ is negligible with respect to $A_{st}^{1}$ as $q_k$ gets large. This proves our claim \eqref{eq:upp-bnd-increment-sigma-k}. 

\end{proof}

\subsection{H\"older continuity}

 We shall use the following notation, valid for $\ga\in(0,1)$, a time horizon $t\in [0,T]$ and a function from $[0,t]$ to $\R^{m}$:
\begin{equation}
\|f\|_{\ga,t}
:= \sup_{0\le s < u \le t} \frac{|\delta f_{st}|}{|u-s|^{\ga}},
\quad\text{where}\quad
\delta f_{st} = f_{t}-f_{s}.
\end{equation}
Then, we have the following result, which is our first main objective in this section.

\begin{proposition} \label{th2}
Suppose that $\si$ satisfies Hypothesis \ref{hyp4.1} and that our noise $x$ satisfies Hypotheses~\ref{hyp:reg-x-gamma-gamma1} and~\ref{hyp:roughness}. We also assume that $\ga+\ka>1$.
Then, the function $y$ given in Proposition~\ref{th1} belongs to $\mathcal{C}^\gamma([0,T]; \R^m)$.
\end{proposition}

\begin{proof}
Case (A) of Proposition \ref{th1} is trivially obtained by Young integration techniques.
Hence it suffices to assume that $y$ satisfies condition (B) in Proposition \ref{th1}.
 Consider first $s=\la_{k}$ and $t=\la_{l}$ with $k<l$. We start by decomposing the increments $|\delta y_{st}|$ as follows
\begin{equation*}
\lln \delta y_{st}  \rrn
\le
\sum_{j=k}^{l-1} \lln \delta y_{\la_{j} \la_{j+1}} \rrn.
\end{equation*}
Then owing to Proposition \ref{prop:bound-diff-sigma-k-2} we have $\la_{k+1}-\la_{k}\le  c_{x,\ep_{2}}  2^{- q_{k} (\al-\ep_{2})}$ for $k$ large enough. We can thus apply Corollary \ref{cor:a-priori-bnd-yn-larger-intv}, which yields
\begin{equation}\label{eq:global-increments-yn-1}
\lln \delta y_{st}  \rrn
\le
\sum_{j=k}^{l-1} \lln \delta y_{\la_{j} \la_{j+1}}  \rrn
\le
c_{5,x}  \sum_{j=k}^{l-1}  2^{-q_{j}\ka_{\ep_{2}}^{-}} |\la_{j+1} - \la_{j}|^{\ga}.
\end{equation}
Furthermore,  inequality \eqref{eq:low-bnd-increment-sigma-k} entails:
\begin{equation*}
2^{-\frac{q_{j} (1-\ka)}{\ga}} \lesssim c_{7,x}^{-1} \lp  \la_{j+1} - \la_{j} \rp
\quad\Longrightarrow\quad
2^{-q_{j} \ka_{\ep_{2}}^{-}} \le 
(c_{7,x})^{-\frac{\ga \ka_{\ep_{2}}^{-}}{1-\ka}}
 \lp  \la_{j+1} - \la_{j} \rp^{\frac{\ga \ka_{\ep_{2}}^{-}}{1-\ka}}.
\end{equation*}
Plugging this information into \eqref{eq:global-increments-yn-1} and setting $c_{10,x}=c_{5,x} (c_{7,x})^{-\frac{\ga \ka_{\ep_{2}}^{-}}{1-\ka}} $,
 we end up with:
\begin{equation*}
\lln \delta y_{st}  \rrn
\le
c_{10,x} 
\sum_{j=k}^{l-1}   |\la_{j+1} - \la_{j}|^{\mu_{\ep_{2}}},
\quad\text{with}\quad
\mu_{\ep_{2}}= \ga\lp 1+  \frac{\ka_{\ep_{2}}^{-}}{1-\ka} \rp.
\end{equation*}
We now wish the exponent $\mu_{\ep_{2}}$ to be of the form $\mu_{\ep_{2}}=1+\ep_{3}$ with $\ep_{3}>0$. Since $\ka_{\ep_{2}}^{-}$ is arbitrarily close to $\ka$, it is readily checked that this can be achieved as long as $\ga+\ka>1$. Recalling that $s=\la_{k}$ and $t=\la_{l}$, one can thus recast the previous inequality as
\begin{equation*}
\lln \delta y_{st}  \rrn
\le
c_{10,x} 
\sum_{j=k}^{l-1}   |\la_{j+1} - \la_{j}|^{1+\ep_{3}}
\le
c_{10,x} |\la_{l} - \la_{k}|^{1+\ep_{3}}
\le 
c_{10,x}  \, \tau^{1+\ep_{3}-\ga} |t-s|^{\ga},
\end{equation*}
which is consistent with our claim.

The general case $s < \la_{k} \le \la_{l} < t$ is treated by decomposing $\delta y_{st}$ as
\begin{equation*}
\delta y_{st} = \delta y_{s \la_{k}}+ \delta y_{\la_{k} \la_{l}}  + \delta y_{\la_{l}t}.
\end{equation*}
Then resort to \eqref{eq:a-priori-bnd-yn-whole-interval-lambda} in order to bound $\delta y_{s \la_{k}}$ and $\delta y_{\la_{l}t}$.
\end{proof}

 The next proposition says that if (B) holds, the function  $y$ can be obtained as the limit of a suitable sequence of Riemann sums.

\begin{proposition}\label{prop:riemann-sums}
Let $y$ be the function  given in Proposition~\ref{th1}.
For all $0\le s < t \le T$, let $\Pi_{st}$ be the set of partitions of $[s,t]$, denoted generically by $\pi=\{ s=t_{0}<\cdots<t_{m}=t$\}. For $\ep>0$ arbitrarily small, define
\begin{equation*}
\Pi_{st}^{\ep}
=
\lcl
\pi\in\Pi_{st}; \text{ there exists }
\js \text{ such that } t_{\js} < \tau \le t_{\js+1} \text{ and } \eta \le |\tau -t_{\js}| \le 2\eta
\rcl,
\end{equation*}
where $\eta=c_{x} \ep^{1/\ga}$ for a strictly positive constant $c_{x}$. Then under the conditions of Proposition~\ref{th2}, one can find $\pi\in\Pi_{st}^{\ep}$ such that:
\begin{equation}\label{eq:lim-riemann-sums}
\lln \ist \si(y_{u}) \, dx_{u}  - \sum_{t_{j}\in\pi}  \si(y_{t_{j}}) \, \delta x_{t_{j}t_{j+1}} \rrn
\le
\ep.
\end{equation}
\end{proposition}

\begin{proof} Consider a partition $\pi$ lying in $\Pi_{st}^{\ep}$, and set $S_{\pi}=\sum_{t_{i}\in\pi}  \si(y_{t_{i}}) \, \delta x_{t_{i}t_{i+1}}$. Since $y_{u}=0$ for $u\ge\tau$, it is worth noting that we also have
\begin{equation*}
S_{\pi} = S_{\pi^{*}} + \si(y_{t_{\js}}) \, \delta x_{t_{\js}t_{\js+1}},
\quad\text{where}\quad
S_{\pi^{*}} \equiv \sum_{j < \js} \si(y_{t_{j}}) \, \delta x_{t_{j}t_{j+1}}.
\end{equation*}
Then we can write
\begin{equation*}
\lln \delta y_{st} - S_{\pi} \rrn \le 
\lln \delta y_{st_{\js}} - S_{\pi^{*}} \rrn + \lln \delta  y_{t_{\js} \tau} \rrn
+ \lln \si(y_{t_{\js}}) \, \delta x_{t_{\js}t_{\js+1}}  \rrn
:= I_{1} + I_{2} + I_{3}.
\end{equation*}
We now bound separately  the 3 terms on the right hand side above. For the term $I_{2}$ we have
\begin{equation*}
I_{2} \le \| y \|_{\ga} |\tau - t_{\js}|^{\ga} \le c_{x}  (2\eta)^{\ga}.
\end{equation*}
We can obviously choose a constant $c_{x}$ such that if $\eta=c_{x} \ep^{1/\ga}$, then $I_{2}\le\frac{\ep}{3}$. Thanks to the same kind of elementary considerations, we can also make the term $I_{3}$ smaller than $\frac{\ep}{3}$. In order to bound $I_{1}$, we invoke the fact that $|\tau-t_{\js}|\ge \eta$ and we set
\begin{equation*}
Q_{\eta} = \inf \lcl |y_{s}| : \, s < \tau - \eta  \rcl.
\end{equation*}
Observe that $Q_{\eta}>0$. In addition, by Hypothesis \ref{hyp4.1} (ii), $\si$ is  differentiable and locally H\"older continuous of order $\frac 1\gamma -1$ on $[Q_{\eta},\infty)$. By usual convergence of Riemann sums for Young integrals, we thus have
\begin{equation*}
\lim_{\pi\in\Pi_{s t_{\js}}, |\pi|\to 0} I_{1}
=
\lim_{\pi\in\Pi_{s t_{\js}}, |\pi|\to 0} \lln \delta y_{st} - S_{\pi} \rrn 
= 0.
\end{equation*}
Therefore we can choose $|\pi|$ so that $I_{1}\le\frac{\ep}{3}$. Putting together our upper bounds on $I_{1}$, $I_{2}$ and $I_{3}$, the proof of \eqref{eq:lim-riemann-sums} is now finished.
\end{proof}

Finally we can summarize the considerations of this section into the following theorem.

\begin{theorem}\label{thm:exist-multidim}
Consider equation \eqref{eq:sde-power-approx}, and let $T$ be a given strictly positive time horizon. We suppose that Hypothesis \ref{hyp4.1} holds for the coefficient $\si$, and that Hypothesis \ref{hyp:reg-x-gamma-gamma1} and \ref{hyp:roughness} are satisfied for our noise $x$. Then, there exists a continuous function $y$ defined on $[0,T]$ and an instant $\tau \le T$, such that one of the following two possibilities holds:  

\noindent
\emph{(A)}
$\tau =T$, $y$ is nonzero on $[0,T]$,  $y\in\cac^{\ga}([0,T];\R^m)$   and $y$ solves equation  \eqref{eq:sde-power-approx} on $[0,T]$, where the integrals $\int \si^{j}(y_{u}) \, dx_{u}^{j}$ are understood in the usual Young sense.

\noindent
\emph{(B)}
We have $\tau<T$. Then for any $t<\tau$, the path $y$ sits in $\cac^{\ga}([0,T];\R^m)$   and $y$ solves equation  \eqref{eq:sde-power-approx} on $[0,T]$, where the integrals $\int \si^{j}(y_{u}) \, dx_{u}^{j}$ are understood as in Proposition~\ref{prop:riemann-sums}.
Furthermore,  $y_s \not=0$ on   $[0,\tau)$, $\lim_{t\rightarrow \tau} y_t=0$ and $y_t=0$ on the interval $[\tau,T]$.
\end{theorem}

\section{Application to fractional Brownian motion}

 Let $B^{H}=\{B_{t}^{H},$ $t\in \lbrack 0,T]\}$ be a standard $d$-dimensional
fractional Brownian motion with \ the Hurst parameter $H\in (\frac{1}{2},1)$
defined on a complete probability space $(\Omega ,\mathcal{F},P)$, that is,
the components of $B^{H}$ are independent centered Gaussian processes with
covariance 
\[
 \mathbf{E}(B_{t}^{H,i}B_{s}^{H,i})=\frac{1}{2}\left( \left| t\right|
^{2H}+\left| s\right| ^{2H}-\left| t-s\right| ^{2H}\right),
\]%
for any $s,t \in [0,T]$.
It is clear that $ \mathbf{E} |B^H_t -B^H_s|^2 = d |t-s|^{2H}$, and, as a consequence, the trajectories of $B^H$ are $\ga$-H\"older continuous for any $\ga <H$.
  Consider
the  $m$-dimensional stochastic differential equation
\begin{equation}  \label{k2}
X_{t}=x_{0}+ \sum_{j=1}^d \int_{0}^{t}\sigma^j
(X_{s})dB_{s}^{H,j} ,\;0\leq t\leq T, 
\end{equation}%
where  $x_{0} \in \R^m$.  
 If $\sigma $ is H\"{o}lder
continuous of order   $\ka >\frac{1}{H}-1$, then, there exists a solution $%
X $ which has H\"{o}lder continuous trajectories of order $\ga $, for any $%
\ga  <H $.  This was proved by Lyons in  \cite{Ly}  using the Young's integral and $p$-variation estimates.
An extension of this result  where there is a measurable drift with linear growth was given by 
  Duncan and Nualart in  \cite{DN}.   Under this weak assumption of $\sigma$ we cannot expect the uniqueness of a solution, which requires $\sigma$ to be differentiable with partial derivatives H\"older continuous of order larger than  $\frac{1}{H}-1$ (see \cite{Ly,NR}).

The results proved in the previous sections  allow us to construct  examples of existence of solutions for equation
(\ref{k2}), when $\si$ is H\"older continuous of order $\ka$ and $\ka < \frac 1H-1$.

\begin{example}
Suppose that $m=d=1$, $x_0=0$  and $\sigma(\xi) =C|\xi|^\ka$, with $\ka <\frac 1H-1$. Then, the process
\[
X_t= \phi^{-1} (B^H_t),
\]
where $\phi(\xi) = \int_0^\xi \frac{dx}{\sigma(x)}$ satisfies equation (\ref{k2}), where the integral is a path-wise integral defined in Proposition  \ref{prop:integral}. Indeed, it suffices to show that the assumptions of Theorem \ref{th1v1} hold. Taking into account that  $\phi^{-1}$ satisfies
\[
{\rm sgn} (\phi^{-1} (\xi)) | (\phi^{-1} (\xi)) |^{1-\ka} = C(1-\ka) \xi,
\]
for any $\xi \in \R$, we get  $| \phi^{-1} (\xi))|  = [C(1-\ka)] ^{\frac 1{1-\ka}} |\xi| ^{\frac 1{1-\ka}}$.
Therefore, for any $\eta <1-\ka$,
\[
\mathbf{E} \int_0^T  |\phi^{-1} (B^H_s)|  ^{-   \eta } ds
=[C(1-\ka)] ^{-\frac \eta{1-\ka}}  \mathbf{E} \int_0^T  |B^H_s|  ^{-\frac \eta{1-\ka}}ds  <\infty.
\]
This implies  $ \int_0^T  |\phi^{-1}  (B^H_s)|  ^{-   \eta } ds <\infty$ almost surely, and we can apply Theorem \ref{th1v1}.
\end{example}

\begin{example}
Consider equation  (\ref{k2}) in the multidimensional case, with $x_0\not =0$. Suppose that  each component $\sigma^j $   satisfies Hypothesis \ref{hyp4.1} with $\ka< \frac 1H-1$ and observe that $B^H$ satisfies Hypotheses \ref{hyp:reg-x-gamma-gamma1} and \ref{hyp:roughness}.
 Then, we can apply Theorems \ref{th1} and  \ref{th2}, and conclude that there exist a  stochastic process $X$  such that, if
\begin{equation*}
\tau   = \inf\lcl t >0 : \, X_{t} = 0 \rcl \wedge T,
\end{equation*}
then,
\[
X_{t} =\left( x_0 +  \sum_{j=1}^{d}\int_{0}^{t} \si^{j}(X_{s}) \, dB_{s}^{H,j}\right) \mathbf{1}_{[0,\tau)} (t),
\]
where for $t<\tau$,  the stochastic integral is understood as  a path-wise Young integral. Moreover,  
the process  $X$ satisfies
$X\in\cac^{\ga}([0,T];\R^m) $ for any $\ga<H$.

\end{example}

\end{document}